\documentclass[11pt]{article}
\usepackage{amsmath,amsfonts,amssymb,amsthm,color}
\usepackage{txfonts}%
\usepackage{graphicx}
\usepackage{longtable}
\usepackage{bm}
\usepackage{rotating}
\usepackage{diagbox}

\usepackage{url}

\theoremstyle{plain}
\newtheorem{theorem}{Theorem}[section]
\newtheorem{lemma}[theorem]{Lemma}

\newtheorem{proposition}[theorem]{Proposition}
\theoremstyle{definition}
\newtheorem{definition}[theorem]{Definition}
\theoremstyle{remark}
\newtheorem{remark}[theorem]{Remark}

\newcommand{\figcaption}[1]{\def\cptype{Figure} \caption{#1}}
\newcommand{\tblcaption}[1]{\def\cptype{Table} \caption{#1}}
\newcommand{\figref}[1]{Figure \ref{#1}}
\newcommand{\tabref}[1]{Table \ref{#1}}
\newcommand{\eq}[1]{(\ref{#1})}

\newcommand{\kerz}{\operatorname{\ker_{\mathbb Z}}}
\newcommand{\swap}{\leftrightarrow}
\newcommand{\swapj}[1]{\stackrel{#1}{\leftrightarrow}}
\newcommand{\swapoperation}[1]{\stackrel{#1}{\longleftrightarrow}}
\newcommand{\ii}{\iota}  

\bmdefine{\Bz}{z}
\bmdefine{\Bx}{x}
\bmdefine{\Bt}{t}

\newcommand{\twoelementcolumn}[2]{\renewcommand\arraystretch{0.6}{\small \left[\!\begin{tabular}{c}$#1$\\$#2$\end{tabular}\!\right]}}

\begin{document}

\title{Markov degree of the Birkhoff model}

\author{
Takashi Yamaguchi\thanks{Graduate School of Information Science and Technology, University of Tokyo}, \ 
Mitsunori Ogawa\footnotemark[1] \ 
and Akimichi Takemura\footnotemark[1]\ \thanks{JST CREST}
}
\date{November, 2013}
\maketitle

\begin{abstract}
We prove the conjecture by Diaconis and Eriksson (2006) that the Markov degree of the Birkhoff model is three.
In fact, we prove the conjecture in a  generalization of the Birkhoff model, where each voter is asked to rank a fixed number, say $r$, of candidates among all candidates.
We also give an exhaustive characterization of Markov bases for small $r$.
\end{abstract}


\section{Preliminaries}
\label{sec:intro}
Diaconis and Eriksson \cite{diaconis-eriksson} conjectured that the Markov degree of the
Birkhoff model is three, i.e., the toric ideal associated with the Birkhoff model is generated by binomials of degree at most three. In this paper we give a proof of this conjecture in
a  generalization of the Birkhoff model, 
where each voter is asked to rank a fixed number 
of most preferred candidates among all candidates.  Our proof is based on arguments
of Jacobson and Matthews \cite{jacobson-matthews} for Latin squares.  
The set of Latin squares is a particular
fiber in our setting and our result is also a generalization of \cite{jacobson-matthews}.
See \cite{aht-book} for terminology of algebraic statistics and toric ideals used in this paper.

Consider an election, where there are $n$ candidates and $N$ voters.
Each voter is asked to give $r$ ($1\le r \le n$) preferred candidates and to rank them.
For example, let $n=5, r=3$ and let the candidates be $a,b,c,d,e$.
A vote  $(a,c,d)$ by a voter means that he/she ranks $a$ first, $c$ second and
$d$ third. For a positive integer $m$, denote $[m]=\{1,\dots,m\}$. 
When the candidates are labeled as $1,\dots,n$, the set of
possible votes is 
\[
S_{n,r} = \{ \sigma = (\sigma(1),\dots,\sigma(r)) \mid \sigma: \text{injection from $[r]$ to $[n]$}\}, \qquad |S_{n,r}|=\frac{n!}{(n-r)!},
\]
where $\sigma(j)$ denotes the candidate chosen in the $j$-th position in the vote $\sigma=(\sigma(1),\dots,\sigma(r))$.
Let $\psi_{jk}, \ j\in [r], k\in [n]$, be positive parameters and define a probability distribution over $S_{n,r}$ by
\begin{equation}
\label{eq:n-r-Birkhoff}
p(\sigma)=\frac{1}{Z}\prod_{j=1}^r \psi_{j \sigma(j)}, \qquad Z=\sum_{\sigma\in S_{n,r}} 
\prod_{j=1}^r \psi_{j \sigma(j)}.
\end{equation}
If $\psi_{jk}$ is large, then the candidate $k$ is likely to be ranked in the $j$-th position.
When $r=n$, this model is the Birkhoff model (\cite{diaconis-eriksson}, \cite{sturmfels-welker}). 
In this paper we call \eqref{eq:n-r-Birkhoff} an $(n,r)$-Birkhoff model.
The sufficient statistic of the $(n,r)$-Birkhoff model consists of numbers of times 
the candidate $k$ is ranked in the $j$-th position, $j\in [r], k\in[n]$.
We denote the sufficient statistic as $(t_{jk})_{j\in [r], k\in [n]}$.

Define a 0-1 matrix $A=A_{n,r}$ of size $rn \times (n!/(n-r)!)$, called a configuration matrix for the $(n,r)$-Birkhoff model, 
whose columns are labeled by $\sigma\in S_{n,r}$ and rows are labeled  by $(j,k)=(\text{position},\text{candidate})$, such that the $((j,k),\sigma)$-element of $A$ is one if and only if $\sigma(j)=k$. 
For example, for $n=4$, $r=3$, the configuration matrix $A_{4,3}$ with labels for its rows and columns is

\begin{equation}
\label{nr=43}
\begin{array}{ccccccccccccccccccccccccc}
&\ \begin{rotate}{90}(123)\end{rotate}&
\begin{rotate}{90}(124)\end{rotate}&
\begin{rotate}{90}(132)\end{rotate}&
\begin{rotate}{90}(134)\end{rotate}&
\begin{rotate}{90}(142)\end{rotate}&
\begin{rotate}{90}(143)\end{rotate}&
\begin{rotate}{90}(213)\end{rotate}&
\begin{rotate}{90}(214)\end{rotate}&
\begin{rotate}{90}(231)\end{rotate}&
\begin{rotate}{90}(234)\end{rotate}&
\begin{rotate}{90}(241)\end{rotate}&
\begin{rotate}{90}(243)\end{rotate}&
\begin{rotate}{90}(312)\end{rotate}&
\begin{rotate}{90}(314)\end{rotate}&
\begin{rotate}{90}(321)\end{rotate}&
\begin{rotate}{90}(324)\end{rotate}&
\begin{rotate}{90}(341)\end{rotate}&
\begin{rotate}{90}(342)\end{rotate}&
\begin{rotate}{90}(412)\end{rotate}&
\begin{rotate}{90}(413)\end{rotate}&
\begin{rotate}{90}(421)\end{rotate}&
\begin{rotate}{90}(423)\end{rotate}&
\begin{rotate}{90}(431)\end{rotate}&
\begin{rotate}{90}(432)\end{rotate}\\
(1,1)&1&1&1&1&1&1&0&0&0&0&0&0&0&0&0&0&0&0&0&0&0&0&0&0\\
(1,2)&0&0&0&0&0&0&1&1&1&1&1&1&0&0&0&0&0&0&0&0&0&0&0&0\\
(1,3)&0&0&0&0&0&0&0&0&0&0&0&0&1&1&1&1&1&1&0&0&0&0&0&0\\
(1,4)&0&0&0&0&0&0&0&0&0&0&0&0&0&0&0&0&0&0&1&1&1&1&1&1\\
(2,1)&0&0&0&0&0&0&1&1&0&0&0&0&1&1&0&0&0&0&1&1&0&0&0&0\\
(2,2)&1&1&0&0&0&0&0&0&0&0&0&0&0&0&1&1&0&0&0&0&1&1&0&0\\
(2,3)&0&0&1&1&0&0&0&0&1&1&0&0&0&0&0&0&0&0&0&0&0&0&1&1\\
(2,4)&0&0&0&0&1&1&0&0&0&0&1&1&0&0&0&0&1&1&0&0&0&0&0&0\\
(3,1)&0&0&0&0&0&0&0&0&1&0&1&0&0&0&1&0&1&0&0&0&1&0&1&0\\
(3,2)&0&0&1&0&1&0&0&0&0&0&0&0&1&0&0&0&0&1&1&0&0&0&0&1\\
(3,3)&1&0&0&0&0&1&1&0&0&0&0&1&0&0&0&0&0&0&0&1&0&1&0&0\\
(3,4)&0&1&0&1&0&0&0&1&0&1&0&0&0&1&0&1&0&0&0&0&0&0&0&0
\end{array}
\end{equation}

Let $x(\sigma) \in {\mathbb N}=\{0,1,\dots\}$ be the frequency of voters choosing a vote $\sigma \in S_{n,r}$ and let $\Bx=\{ x(\sigma)\mid \sigma\in S_{n,r}\}$ be the vector of frequencies.
Then $\Bt = A_{n,r}\Bx$ is the sufficient statistic vector.
For a given $\Bt$, ${\cal F}_\Bt=\{ \Bx\in {\mathbb N}^{|S_{n,r}|} \mid A\Bx = \Bt\}$ is the $\Bt$-fiber.

Let $K$ be any field and let $K[\{p(\sigma)$, $\sigma\in S_{n,r}\}]$ be the polynomial ring in the indeterminates  $p(\sigma), \sigma\in S_{n,r}$.
Similarly let $K[\{\psi_{jk}, j\in [r], k\in [n]\}]$ be the polynomial ring in the indeterminates $\psi_{jk}, j\in [r], k\in [n]$.
Let 
\[\pi_{n,r}: K[\{p(\sigma)\mid \sigma\in S_{n,r}\}] \rightarrow 
K[\{\psi_{jk}, j\in [r], k\in [n]\}]
\]
be a homomorphism defined by
\[
\pi_{n,r}: p(\sigma) \mapsto \prod_{j=1}^r \psi_{j \sigma(j)}.
\]
Then the toric ideal $I_A=I_{A_{n,r}}$ for the $(n,r)$-Birkhoff model is the kernel of $\pi_{n,r}$. 
Moves for $A_{n,r}$ are the elements of the integer kernel $\kerz A_{n,r}=\{ \Bz\in {\mathbb Z}^{|S_{n,r}|} \mid A\Bz=0\}$ of $A_{n,r}$.

Note that if a voter ranks $r=n-1$ most preferred candidates, then he/she automatically ranks the last candidate. 
It is easy to see that the configuration matrix $A_{n,n-1}$  for the $(n,n-1)$-Birkhoff model and the  configuration matrix $A_{n,n}$ for the Birkhoff model have the same number of columns and their integer kernels are the same: $\kerz A_{n,n-1}=\kerz A_{n,n}$.

\section{Main result and its proof}
\label{sec:main}
The main result of this paper is the following theorem.

\begin{theorem}
\label{thm:main}
For $r\ge 2$ and $n\ge 3$, the toric ideal $I_A$ for the $(n,r)$-Birkhoff model is generated by
binomials of degree two and three. 
\end{theorem}
For $r=1$ or $r=n=2$, the toric ideal $I_A$ is trivial.
For $r\ge 2$ and $n\ge 3$, any set of generators for $I_A$ contains a binomial of degree three. 
In the terminology of algebraic statistics, Theorem \ref{thm:main} states that the Markov degree of the $(n,r)$-Birkhoff model is three for $r\ge 2$ and $n\ge 3$.

The rest of this section is devoted to a proof of this theorem.
We define some notation and terminology for our proof, mainly following \cite{jacobson-matthews}.
Candidates are denoted either by letters $a,b,c,\ldots$ or by numbers $1,\dots,n$.
The set of $n$ candidates is denoted by $[n]$, using numbers.

First we give the definition for ``valid'' votes and definitions for two kinds of ``invalid votes''.
Our proof will be based on the idea of swapping candidates between two votes.
\begin{definition}
\label{def:proper}
An $r\times n$ integer matrix $V=(v_{jk})$ is a {\em proper vote} if $v_{jk} \in \{ 0,1 \} , \forall j,k$, every row sum is one, and every column sum is zero or one.
A {\em proper dataset} of $N$ votes is the multiset of $N$ proper votes.
\end{definition}
\begin{definition}
\label{def:improper}
An $r\times n$ integer matrix $V=(v_{jk})$ is an {\em improper vote} if every row sum is one, every column sum is zero or one, and
there exists a unique cell $(j^* , k^*) \in [r] \times [n]$ such that
\[
v_{j^* k^*}=-1, \quad v_{j k} \in \{ 0, 1 \} , \ \forall (j, k) \neq (j^* , k^*).
\]
An {\em improper dataset} of $N$ votes is the multiset of $r\times n$ integer matrices $V^{(1)}=(v_{jk}^{(1)}), \ldots, V^{(N)}=(v_{jk}^{(N)})$ such that one of them is an improper vote, the others are proper votes, and $\sum_{i=1}^N v_{jk}^{(i)} \geq 0, \forall j,k$. 
\end{definition}
\begin{definition}
\label{def:collision}
An $r\times n$ integer matrix $V=(v_{jk})$ is a {\em vote with collision} if $v_{jk} \in \{ 0,1 \} , \forall j,k$, every row sum is one and there exists a unique candidate $k^* \in [n]$ such that 
\[
\sum_{j=1}^r v_{j k^*}=2, \quad \sum_{j=1}^r v_{j k} \in \{ 0, 1 \} , \ \forall k\neq k^*.
\]
In this case we also say that the vote $V$ contains a {\em collision} or the candidate $k^*$ {\em collides} in $V$.
\end{definition}
\begin{definition}
\label{def:imp_col}
An $r\times n$ integer matrix $V=(v_{jk})$ is an {\em improper vote with collision} if every row sum is one, there exists a unique cell $(j^* , k^*) \in [r] \times [n]$ such that
\[
v_{j^* k^*}=-1, \quad v_{j k} \in \{ 0, 1 \} , \ \forall (j, k) \neq (j^* , k^*),
\]
and there exists a unique candidate $k^{**} \in [n]$ such that 
\[
\sum_{j=1}^r v_{j k^{**}}=2, \quad \sum_{j=1}^r v_{j k} \in \{ 0, 1 \} , \ \forall k\neq k^{**}.
\]
\end{definition}
We call a multiset $D$ of $r\times n$ integer matrices a {\em dataset} if each matrix in $D$ is one of the votes defined in Definitions \ref{def:proper}--\ref{def:imp_col}.

As in Section \ref{sec:intro},  we often denote votes by row vectors.
For proper votes and votes with collision, we denote them by $r$-dimensional row vectors whose $j$-th entry is the candidate ranked in the $j$-th position for each $j\in [r]$.
For improper votes we define their row vector representation as follows.
Let $V=(v_{jk})$ be an improper vote with $v_{j^* k} = v_{j^* k^{\prime}}=1, v_{j^* k^{\prime \prime}}=-1$.
We denote $V$ by an $r$-dimensional row vector whose $j$-th entry is the candidate ranked in the $j$-th position for each $j \in [r]$ with $j \neq j^*$ and the $j^*$-th entry is $k + k^{\prime} - k^{\prime \prime}$.
Here, $k + k^{\prime} - k^{\prime \prime}$ is just a symbol and we call it an {\em improper element}.
The following vectors are examples of a proper vote, a vote with collision and an improper vote, respectively:
\[
(b, a, c, d), \quad (b, a, c, a), \quad (a, d, a, b+c-a).
\]
We also define the row vector representation for improper votes with collision in the similar manner.

Several kinds of datasets were defined as the multiset of integer matrices above.
For these datasets we use their matrix representation.
The matrix representation $\bar{D}$ for a dataset $D=\{ V^{(1)}, \ldots, V^{(N)} \}$ is an $N \times r$ matrix whose $i$-th row is the row vector representation of vote $V^{(i)}$ for each $i\in [N]$.
Although the order of the rows of $\bar{D}$ are arbitrary, this matrix representation is convenient for our proof.
When there is no confusion, $\bar{D}$ is also called a dataset.
Latin squares are also of this form.
They are tables with $N=n=r$ such that each candidate appears exactly once in each row and column.
An example of improper dataset $I$ and its matrix representation $\bar{I}$ is as follows:
\begin{eqnarray*}
\small
I=\left\{
\left[
\begin{tabular}{ccc}
$0$&$0$&$1$\\
$0$&$0$&$1$\\
$1$&$1$&$-1$
\end{tabular}
\right],
\left[
\begin{tabular}{ccc}
$0$&$1$&$0$\\
$1$&$0$&$0$\\
$0$&$0$&$1$
\end{tabular}
\right],
\left[
\begin{tabular}{ccc}
$1$&$0$&$0$\\
$0$&$1$&$0$\\
$0$&$0$&$1$
\end{tabular}
\right]
\right\},
\quad
\bar{I}=\left[
\begin{tabular}{ccc}
$c$&$c$&$a+b-c$\\
$b$&$a$&$c$\\
$a$&$b$&$c$
\end{tabular}
\right] .
\end{eqnarray*}
In the following, when we display a dataset, we mainly use its matrix representation.

We now introduce some operations for datasets. 
Let $D=\{ V^{(1)},...,V^{(N)}\}$ be a dataset. 
Consider a pair of distinct votes in $D$, say $V^{(i_1)}=(\nu_{jk}^{(i_1)})$ and $V^{(i_2)}=(\nu_{jk}^{(i_2)})$. 
A {\em swap} $\{i_1,i_2 \} : k_1 \swapj{j^*} k_2$ for $D$ is an operation transforming $D$ into $D^{\prime}$, where
\begin{eqnarray*}
&& D^{\prime} = ( D \setminus \{ V^{(i_1)} , V^{(i_2)} \} )
\cup \{ \tilde{V}^{(i_1)}, \tilde{V}^{(i_2)} \},  \quad
\tilde{V}^{(i_1)}=(\tilde{v}^{(i_1)}_{jk}), \tilde{V}^{(i_2)}=(\tilde{v}^{(i_2)}_{jk}) ,
\\
&& \tilde{v}^{(i)}_{jk}=
\begin{cases} 
v^{(i)}_{jk}+1,\,\, & \quad (i,j,k)=(i_1, j^*, k_2), (i_2, j^*, k_1), \\
v^{(i)}_{jk}-1,\,\, & \quad (i,j,k)=(i_1, j^*, k_1), (i_2, j^*, k_2), \\
v^{(i)}_{jk},  \,\, & \quad \text{otherwise}. 
\end{cases}
\end{eqnarray*}
In general, $D^{\prime}$ is not a dataset, because $\tilde{V}^{(i)}, i=i_1,i_2$,  may not be a vote defined above.  
If $D^{\prime}$ is also a dataset, the swap is called an {\em applicable swap}. 
Since only applicable swaps appear in our proof, applicable swaps are called merely swaps, hereafter. 
The swap operation does not alter the sufficient statistic of the dataset.
Note that the swap may cause a new collision or a new improper element in $\tilde{V}^{(i)}, i=i_1,i_2$.

The matrix representation of datasets helps intuitive understanding and manipulation of the swap operation. 
The definition of the swap operation shows that the $j^*$-th row of $\tilde{V}^{(i_1)}$ is the sum of the $j^*$-th row of $V^{(i_1)}$ and a row vector with $0$ entries except $-1$ for the $k_1$-th entry and $1$ for the $k_2$-th entry. 
Similarly, the $j^*$-th row of $\tilde{V}^{(i_2)}$ is the sum of the $j^*$-th row of $V^{(i_2)}$ and a row vector with $0$ entries ~~~~except $1$ for the $k_1$-th entry and $-1$ for the $k_2$-th entry. 
Hence, the matrix representation $\bar{D}^{\prime}$ of $D^{\prime}$ is the sum of the matrix representation $\bar{D}$ of $D$ and the matrix with $0$ entries except $k_2-k_1$ for the $(i_1, j^*)$-entry and $k_1-k_2$ for the $(i_2, j^*)$-entry where $k_2-k_1$ and $k_1-k_2$ are symbols.

For illustration, we show an example of swap in the matrix representation of datasets.
Let $\bar{P} = (p_{ij})$ be a proper dataset in its matrix representation with $p_{11}=a, p_{21}=b$, $a\neq b$ and consider a swap $\{1,2 \} : a \swapj{1} b$.
By adding
\[
\left[
\begin{tabular}{c}
$b-a$\\
$a-b$
\end{tabular}
\right]
\]
to the submatrix $\twoelementcolumn{a}{b}$ we interchange two candidates $a$ and $b$ as
\begin{equation}\label{eq:swap}
\left[
\begin{tabular}{ccccc}
$a$&$*$&$\cdots$&$*$\\
$b$&$*$&$\cdots$&$*$\\
$*$&$*$&$\cdots$&$*$\\
$\vdots$&$\vdots$&$\vdots$&$\vdots$\\
$*$&$*$&$\cdots$&$*$\\
\end{tabular}
\right]
+ 
\left[
\begin{tabular}{ccccc}
$b-a$&0&$\cdots$&0\\
$a-b$&0&$\cdots$&0\\
0&0&$\cdots$&0\\
$\vdots$&$\vdots$&$\vdots$&$\vdots$\\
0&0&$\cdots$&0\\
\end{tabular}
\right]=
\left[
\begin{tabular}{ccccc}
$b$&$*$&$\cdots$&$*$\\
$a$&$*$&$\cdots$&$*$\\
$*$&$*$&$\cdots$&$*$\\
$\vdots$&$\vdots$&$\vdots$&$\vdots$\\
$*$&$*$&$\cdots$&$*$\\
\end{tabular}
\right],
\end{equation}
where candidates denoted by $*$ are not changed.  
In this case $a$ and $b$ may collide after the swap.
In order to simplify the notation, we sometimes denote the swap by $a \swapj{j} b$ or $a\swap b$.

Let us discuss a sequence of swaps. 
Consider swapping $a$ and $b$ in two different positions $j, j^{\prime}$ in the same $i$-th and $i^{\prime}$-th votes.
In our proof below, we often perform these two swaps sequentially, i.e., 
we swap $a$ and $b$ in the $j$-th position first and then in the $j^{\prime}$-th position. 
We denote this operation as
\[
a\swapj{j}b \swapj{j^{\prime}}a \qquad  \text{or} \qquad \{i,i^{\prime} \}: a\swapj{j}b \swapj{j^{\prime}}a
\]
and call this a {\em double swap}.
The double swap corresponds to the basic move for no three-factor interaction model
(cf.\ \cite{aoki-takemura-anz}). 
As an example, a double swap $a\swapj{1}b \swapj{2}a$, where the second swap causes an improper element, is written as 
\begin{equation}\label{swap}
\left[
\begin{tabular}{cc}
$a$&$b$\\
$b$&$c$
\end{tabular}
\right]
+ 
\left[
\begin{tabular}{cc}
$b-a$&$a-b$\\
$a-b$&$b-a$
\end{tabular}
\right]
=
\left[
\begin{tabular}{cc}
$b$&$ a$\\
$a$&$b+c-a$
\end{tabular}
\right].
\end{equation}

More generally, we consider a sequence of $m$ swaps in positions $j_1, \dots, j_m$, such that two consecutive swaps involve a common candidate,  and denote it as
\begin{equation}
\label{eq:chain-swap}
a_1\swapj{j_1} a_2 \swapj{j_2} \cdots \swapj{j_{m-1}} a_m \swapj{j_m} a_{m+1}
\end{equation}
or indicating the votes as
\begin{equation}
\label{eq:chain-swap1}
\{i,i'\}: a_1\swapj{j_1} a_2 \swapj{j_2} \cdots \swapj{j_{m-1}} a_m \swapj{j_m} a_{m+1} .
\end{equation}
We call \eqref{eq:chain-swap} (or \eqref{eq:chain-swap1}) a {\em chain swap} of length $m$ (even when $a_1=a_{m+1}$, i.e., we do not make a distinction between a chain and a loop).
A chain swap of length one is just a swap.

Suppose that we perform several chain swaps for the same two votes and ignore the order of swaps. 
An even number of swaps on two proper elements at the same position results in no swap and an odd number swaps on two proper elements at the same position results in a single swap.

On the other hand, we need to be careful for swaps involving an improper element.
Let $b+c-a$ be an improper element in the $j$-th position in an improper dataset $\bar{I}$.  
Since the elements of the sufficient statistic of $I$ are assumed to be nonnegative, there is a vote of $\bar{I}$ containing $a$ in the same position as $b+c-a$.  
If we make a swap $a\swap b$ between these two elements, then $b+c-a$ becomes $c$ and $a$ ~~becomes $b$:
\begin{equation}
\label{eq:resolve-improper}
\left[
\begin{tabular}{c}
$b+c-a$\\
$a$
\end{tabular}
\right]
+
\left[
\begin{tabular}{c}
$a-b$\\
$b-a$
\end{tabular}
\right]
=
\left[
\begin{tabular}{c}
$c$\\
$b$
\end{tabular}
\right].
\end{equation}
Similarly $a\swap c$ results in $\twoelementcolumn{b}{c}$.
Note that $\twoelementcolumn{c}{b}$ and $\twoelementcolumn{b}{c}$ are swaps of each other.
Hence the result of several swaps can be regarded as a single swap $a\swap b$ or $a\swap c$.
Although there is an ambiguity between $a\swap b$ or $a\swap c$, the result of a swap between these votes at the $j$-th position is either $\twoelementcolumn{b}{c}$ or $\twoelementcolumn{c}{b}$.

Furthermore we consider a swap between two candidates in $\twoelementcolumn{b+c-a}{d}$, $d\neq a,b,c$.
We allow $b\swap d$ or $c\swap d$ between these two elements.  
After $b\swap d$ we have
\begin{eqnarray}
\label{eq:compatible}
\left[
\begin{tabular}{c}
$b+c-a$\\
$d$
\end{tabular}
\right] \ \rightarrow \ 
\left[
\begin{tabular}{c}
$d+c-a$\\
$b$
\end{tabular}
\right]
\end{eqnarray}
and to $\twoelementcolumn{d+c-a}{b}$ we can make further swaps $c\swap b$ or $d\swap b$.  
The end result of several swaps is one of the following three cases
\[
\left[
\begin{tabular}{c}
$c+d-a$\\
$b$
\end{tabular}
\right],    \ 
\left[
\begin{tabular}{c}
$b+d-a$\\
$c$
\end{tabular}
\right] 
 \ \text{or} \ 
\left[
\begin{tabular}{c}
$b+c-a$\\
$d$
\end{tabular}
\right] .
\]
These three cases correspond to single swaps  $b\swap d$, $c\swap d$ and to no swap to $\twoelementcolumn{b+c-a}{d}$.  

Although there is an ambiguity on the result of chain swaps involving an improper element, the end result of several chain swaps is a set of simultaneous swaps of a subset of positions among the two votes.
We call this a swap operation for a subset of positions among two votes, or simply a {\em swap operation among two votes}.  
When we apply a swap operation to a proper or an improper dataset $D$ for a subset $J$ of positions among two votes $R=\{i,i^{\prime}\}$ and the result is $D^{\prime}$, we denote the operation by a long double sided arrow:
\[
D \swapoperation{R} D^{\prime},
\]
where we omit $J$, because it is often cumbersome to specify $J$.
In this notation we denote a proper dataset by $P$ and an improper dataset by $I$, when we want to clarify the kinds of datasets, instead of $D$.

We now give a proof of Theorem \ref{thm:main} in a series of lemmas.
Let $P$ and $P^{\prime}$ be two proper datasets with the same sufficient statistic.
Our strategy for a proof is to perform swap operations to $P$, involving at most three votes of $P$ at each step, to increase the number of the common elements in $\bar{P}$ and $\bar{P}^{\prime}$.
In each operation, elements at the same position of the three votes of $\bar{P}$ are permuted.
This corresponds to a move of degree at most three.
In fact,  each operation will be further decomposed into a series of swap operations among two votes, which involve intermediate improper datasets.

For the $i$-th vote of $\bar{P}$ and the $i^{\prime}$-th vote of $\bar{P}^{\prime}$
\[
(p_{i1},\dots, p_{ir}),  \quad (p^{\prime}_{i^{\prime}1},\dots, p^{\prime}_{i^{\prime}r}),
\]
let
\begin{equation}\label{eq:coincidences}
C=C_{i,i^{\prime}}=|\{ j \mid p_{ij}=p^{\prime}_{i^{\prime}j}\}|
\end{equation}
be the number of the same candidates in the same positions in these two votes. 
We call $C$ the number of {\em concurrences}.
If $C=r$, then we can remove these two votes from $\bar{P}$ and $\bar{P}^{\prime}$ and consider other $N-1$ votes. 
On the other hand, we will show that, if $C<r$ then we can always increase $C$ by a series of swap operations involving at most three votes of $\bar{P}$.
The $i$-th vote of $\bar{P}$ will eventually coincide with the $i^{\prime}$-th vote of $\bar{P}^{\prime}$.
Then, Theorem \ref{thm:main} is proved by induction on $N$.

Our first lemma concerns resolving collisions.

\begin{lemma}
\label{lem:collision lemma}
Let $\bar{D}$ be a dataset without any improper element and suppose that at least one of the $i$-th and $i^{\prime}$-th votes contains a collision.
If each candidate appears at most twice in these two votes in total, we can resolve all the collisions by a swap operation among these two votes.
\end{lemma}

\begin{remark}
We can prove this lemma based on the normality of the semigroup generated by the configuration matrix $A_{n,r}$ such as $A_{4,3}$ in \eqref{nr=43}.  
The normality follows from results in \cite{ohsugi-hibi-2001}, \cite{sullivant-2006} and \cite{brualdi-dahl}.  
We will discuss this point again in Section \ref{subsec:normality}.
However, we give our own proof of Lemma \ref{lem:collision lemma}, because we will use similar arguments for improper datasets. Arguments based on the normality cannot be applied to improper datasets.
\end{remark}

\begin{proof}
We may assume $i=1$ and $i^{\prime}=2$ and at least the first vote contains a collision.
We first consider the case that there is only one collision in the two votes.
Let $a$ denote the colliding candidate.
Relabeling the positions, without loss of generality, the two votes are displayed as
\[
\left[\begin{tabular}{cccccc}
$a$ & $a$ & $d$ & $*$ & $\cdots$ & $*$\\
$b$ & $c$ & $*$ & $*$ & $\cdots$ & $*$
\end{tabular}
\right],
\]
where $b\neq c$. 
We choose one of the two $a$'s arbitrarily, say in the second position, and make a swap $a\swapj{2}c$ with the following result:
\begin{equation}
\label{eq:send-step-a}
\left[\begin{tabular}{cccccc}
$a$& $c$ & $d$ & $*$ & $\ldots$ & $*$\\
$b$  & $a$ &  $*$  & $*$& $\ldots$ & $*$
\end{tabular}
\right].
\end{equation}
Since $a$ appears at most twice in these two votes in total, $a$ does not collide in the second vote.
However, $c$ might again collide in the first vote, e.g.,
\[
\left[\begin{tabular}{ccccccc}
$a$ & $c$ & $d$ & $c$ & $*$ & $\ldots$ & $*$
\end{tabular}
\right] .
\]
We then make a  swap for $c$, which was in the first vote from the beginning (in this example $c\swapj{4}*$).
If we continue this process, we always have collisions in the first vote.
If this process ends in finite number of steps, then by a chain swap we resolve the collisions of $a$ and subsequent collisions due to swaps.
We claim that this process indeed ends in finite number of steps.
Actually we show a stronger result that no candidate appears twice in this process of resolving collisions.

Suppose otherwise.  
Then there is a candidate, say $\alpha$, which is swapped twice for the first time. 
We consider two cases $\alpha=a$ and $\alpha \neq a$.

Consider the case $\alpha=a$.
The process of swaps is displayed as follows:
\[
a\leftrightarrow c \leftrightarrow s_1  \leftrightarrow \cdots \leftrightarrow s_{l-1}  \leftrightarrow a
\leftrightarrow \cdots \ .
\]
Since the collision always occurs in the first vote, the candidate $a$ was moved from the second vote to the first vote in the swap $s_{l-1}  \leftrightarrow a$.
By \eqref{eq:send-step-a} we have $c=s_{l-1}$, which contradicts the assumption that $\alpha=a$ is the first candidate colliding twice.

Consider the case $\alpha\neq a$.
The process of swaps is displayed as follows:
\begin{eqnarray} \label{eq:case_alpha_different}
a\leftrightarrow c \leftrightarrow s_1  \leftrightarrow \cdots \leftrightarrow s_{l-1} \leftrightarrow \alpha
\leftrightarrow s_{l+1} \leftrightarrow \cdots \leftrightarrow s_{m-1} \leftrightarrow \alpha \leftrightarrow \cdots \ .
\end{eqnarray}
Considering the subprocess of \eq{eq:case_alpha_different} which starts from the first $\alpha$, we can apply the discussion for the $\alpha =a$ and confirm that there exists a contradiction.
We have shown the lemma for the case that there is only one collision.

Now suppose that there are $m$ colliding candidates $a_1,a_2,\ldots,a_m$.
Each of these candidates appears in one of the votes twice.  
Temporarily, we assign different labels, say $a_l^{\prime}, a_l^{\prime \prime}$, $l=2,\ldots,m$, to candidates except for $a_1$, namely, we ignore collisions of $a_2,\ldots,a_m$.
By the above procedure we resolve the collision of $a_1$ and subsequent collisions.
When this procedure is finished, we restore the labels  $a_l^{\prime}, a_l^{\prime \prime} \rightarrow a_l$, $l=2,\dots,m$.
Then some collisions of $a_2,\dots,a_m$ may have been already resolved, but we do not have any new collision.  
Hence, by the above procedure we decrease the number of collisions.
As long as there is a remaining collision, we can repeat this procedure and resolve all the collisions.
\end{proof}

So far we discussed resolving collisions.  
We now consider resolving an improper element by a swap operation among two votes.  
\begin{lemma}\label{lem:getting proper lemma}
Let $\bar{I}$ be an improper dataset containing an element $b+c-a$. 
By a swap operation among two votes, $\bar{I}$ can be transformed to a proper dataset.
\end{lemma}

\begin{proof}
Without loss of generality, assume that the first vote contains $b+c-a$ and the second vote ~~contains $a$.
We can then make a swap $\{1,2\}:a\swap b$, as in \eqref{eq:resolve-improper}.
Here $a$ may collide in the first vote and $b$ may collide in the second vote.
However, both $a$ and $b$ appear at most twice in these two votes.
Hence we can now resolve these possible collisions by Lemma \ref{lem:collision lemma} by a swap operation among these two votes.
\end{proof}

The operation of Lemma \ref{lem:getting proper lemma} is denoted by
\begin{equation}
\label{eq:swap-operation-IP}
I \swapoperation{R} P,
\end{equation}
where $R$ is a set of two votes of $I$.

At this point we make the following two definitions.
\begin{definition}
\label{def:resolvable pair}
We call two votes in Lemma \ref{lem:getting proper lemma} of the form
\[
\begin{tabular}{c}
$i_{\rm im}$\\
$i_{\rm pr}$
\end{tabular}
\left[\begin{tabular}{ccccccc}
$*$& $\cdots$ & $*$ & $b+c-a$ & $*$ & $\cdots$ & $*$\\
$*$& $\cdots$ & $*$ & $a$     & $*$ & $\cdots$ & $*$\\
\end{tabular}
\right]
\]
a {\em resolvable pair}.  
Here $i_{\rm im}$ is an improper vote and $i_{\rm pr}$ is a proper vote.
A resolvable pair is denoted ~~~~~~~~as $[i_{\rm im},i_{\rm pr}]$.
\end{definition}
Note that any improper dataset $\bar{I}$ contains a resolvable pair $[i_{\rm im}, i_{\rm pr}]$ and $R$ in \eqref{eq:swap-operation-IP} is the set of votes of a resolvable pair.

\begin{definition}
\label{def:compatible swap} 
A swap operation among two votes $R=\{i,i'\}$ in $I \swapoperation{R} I^{\prime}$ is {\em compatible} with improper datasets $\bar{I}$ and $\bar{I}^{\prime}$ if there exists a common resolvable pair $[i_{\rm im}, i_{\rm pr}]$ of $\bar{I}$ and $\bar{I}^{\prime}$ such that $R \cap \{i_{\rm im}, i_{\rm pr}\} \neq \emptyset$, or equivalently $|R\cup \{i_{\rm im}, i_{\rm pr}\}|\le 3$.
\end{definition}

\begin{lemma}\label{lem:proper transposition lemma}
Let $P, P^{\prime}$ be two proper datasets with the same sufficient statistic.
Suppose that the $i$-th vote of $P$ and the $i^{\prime}$-th vote of $P^{\prime}$ are different, i.e., $V^{(i)} \neq V^{\prime (i^{\prime})}$, and let $C<r$ in (\ref{eq:coincidences}), be the number of concurrences in these two votes. 
Then, $C$ can be increased by at most three steps of swap operations among two votes of $P$, 
where 1) each intermediate swap operation between two consecutive improper datasets is compatible with them, and
2) if the resulting dataset is improper then its improper vote and the $i$-th vote form a resolvable pair.
\end{lemma}

\begin{proof}
Without loss of generality we consider the first votes of $P$ and $P^{\prime}$.
We consider two disjoint cases.
\begin{description}
\item[Case 1] The same candidate appears in distinct positions in the two votes.\\
Let $b$ be the candidate appearing in the distinct positions in the two votes.
Relabeling the positions, without loss of generality,  let $p_{11}=a$, $p^{\prime}_{11}=b$, $a\neq b$, and $p_{12}=b$. 
Since $P^{\prime}$ contains $b$ in the first position and the sufficient statistic of $P$ and $P^{\prime}$ are the same, $P$ has to contain $b$ in the first position, say $p_{21}=b$.
We now perform a double swap $a\swapj{1}b\swapj{2}a$ to $P$:
\[
\small
\left[
\begin{tabular}{ccccc}
$a$&$b$&$*$&$\cdots$&$*$\\
$b$&$c$&$*$&$\cdots$&$*$\\
$*$&$*$&$*$&$\cdots$&$*$\\
$\vdots$&$\vdots$&$\vdots$&$\vdots$&$\vdots$\\
$*$&$*$&$*$&$\cdots$&$*$\\
\end{tabular}
\right]
\rightarrow
\left[
\begin{tabular}{ccccc}
$b$&$a$&$*$&$\cdots$&$*$\\
$a$&$b+c-a$&$*$&$\cdots$&$*$\\
$*$&$*$&$*$&$\cdots$&$*$\\
$\vdots$&$\vdots$&$\vdots$&$\vdots$&$\vdots$\\
$*$&$*$&$*$&$\cdots$&$*$\\
\end{tabular}
\right], 
\]
where $*$'s are not changed.
By this double swap $C$ is increased. 
If $c=a$, the swap results in a proper dataset.
Otherwise, the swap results in an improper dataset, where $[2,1]$ forms a resolvable pair.
Therefore, $C$ is increased by a process of the form $P \swapoperation{\{ 1,2 \} } P$ or $P \swapoperation{\{ 1,2 \} } I$.

\item[Case 2] Every candidate appearing twice in the two votes appears in the same position.\\
Again let $p_{11}=a$, $p^{\prime}_{11}=b$, $a\neq b$.
The candidate $b$ does not appear in the first vote of $P$ and $a$ does not appear in the first vote of $P^{\prime}$. 
As in Case 1, we can assume $p_{21}=b$.
Since the first vote of $P^{\prime}$ does not contain $a$, the total frequency of the candidate $a$ in $P^{\prime}$ is less than $N$.  
Since the sufficient statistic is common, it follows that there is a vote of $P$ which does not contain $a$.

If the second vote does not contain $a$, we can make a swap $a\swapj{1} b$ among the first two votes and increase $C$ without causing collision.
This process is of the form $P \swapoperation{\{ 1,2 \} } P$.

If the second vote contains $a$, without loss of generality, let $p_{22}=a$ and also assume that the third vote of $P$ does not contain $a$. 
Let $p_{32}=c\neq a$. 
Since $a$ is chosen in the second vote and not chosen in the third vote and both votes have the same number $r$ of candidates, there is a candidate $d$, who is chosen in the third vote but is not chosen in the second vote.  
If $d$ is in the position $j>2$, then by relabeling of positions we assume that $p_{33}=d$.  
Then $\bar{P}$ looks like
\begin{eqnarray*}
\left[
\begin{tabular}{ccccc}
$a$&$*$&$*$&$\cdots$&$*$\\
$b$&$a$&$*$&$\cdots$&$*$\\
$d$&$c$&$*$&$\cdots$&$*$\\
$*$&$*$&$*$&$\cdots$&$*$\\
$\vdots$&$\vdots$&$\vdots$&$\cdots$&$\vdots$\\
$*$&$*$&$*$&$\cdots$&$*$\\
\end{tabular}
\right]
\ \mathrm{or}\ 
\left[
\begin{tabular}{ccccc}
$a$&$*$&$*$&$\cdots$&$*$\\
$b$&$a$&$*$&$\cdots$&$*$\\
$*$&$d(=c)$&$*$&$\cdots$&$*$\\
$*$&$*$&$*$&$\cdots$&$*$\\
$\vdots$&$\vdots$&$\vdots$&$\cdots$&$\vdots$\\
$*$&$*$&$*$&$\cdots$&$*$\\
\end{tabular}
\right]
\ \mathrm{or}\ 
\left[
\begin{tabular}{ccccc}
$a$&$*$&$*$&$\cdots$&$*$\\
$b$&$a$&$*$&$\cdots$&$*$\\
$*$&$c$&$d$&$\cdots$&$*$\\
$*$&$*$&$*$&$\cdots$&$*$\\
$\vdots$&$\vdots$&$\vdots$&$\cdots$&$\vdots$\\
$*$&$*$&$*$&$\cdots$&$*$\\
\end{tabular}
\right] .
\end{eqnarray*}
We perform a swap $\{2,3\}: a\swapj{2} d$ to the second position of the second and third votes:
\begin{eqnarray*}
\small
\left[
\begin{tabular}{ccccc}
$a$&$*$&$*$&$\cdots$&$*$\\
$b$&$d$&$*$&$\cdots$&$*$\\
$d$&$a+c-d$&$*$&$\cdots$&$*$\\
$*$&$*$&$*$&$\cdots$&$*$\\
$\vdots$&$\vdots$&$\vdots$&$\cdots$&$\vdots$\\
$*$&$*$&$*$&$\cdots$&$*$\\
\end{tabular}
\right]
\ \mathrm{or}\ 
\left[
\begin{tabular}{ccccc}
$a$&$*$&$*$&$\cdots$&$*$\\
$b$&$d$&$*$&$\cdots$&$*$\\
$*$&$a$&$*$&$\cdots$&$*$\\
$*$&$*$&$*$&$\cdots$&$*$\\
$\vdots$&$\vdots$&$\vdots$&$\cdots$&$\vdots$\\
$*$&$*$&$*$&$\cdots$&$*$\\
\end{tabular}
\right]
\ \mathrm{or}\ 
\left[
\begin{tabular}{ccccc}
$a$&$*$&$*$&$\cdots$&$*$\\
$b$&$d$&$*$&$\cdots$&$*$\\
$*$&$a+c-d$&$d$&$\cdots$&$*$\\
$*$&$*$&$*$&$\cdots$&$*$\\
$\vdots$&$\vdots$&$\vdots$&$\cdots$&$\vdots$\\
$*$&$*$&$*$&$\cdots$&$*$\\
\end{tabular}
\right] .
\end{eqnarray*}
After the swap the second vote does not contain $a$.
The result is proper if $c=d$ (the middle case) and improper if $c\neq d$.

Now we apply a swap $\{1,2\}: a\swapj{1}b$ for the first position of the first and the second votes and increase $V$.
In the case $c\neq d$, the last swap was performed on an improper dataset, but
it is compatible with the datasets.
Furthermore we can resolve the improper element $a+c-d$ by Lemma \ref{lem:getting proper lemma},  since $[3,2]$ is a resolvable pair.
The process in this case is summarized as $P \swapoperation{\{ 2,3 \} } P \swapoperation{\{ 1,2 \} } P$ or $P \swapoperation{\{ 2,3 \} } I \swapoperation{\{ 1,2 \} } I \swapoperation{\{ 2,3 \} } P$.
\end{description}
This proves the lemma.
\end{proof}

\begin{lemma}
\label{lem:improper transposition lemma}
Let $I$ be an improper dataset and $P^{\prime}$ be a proper dataset with the same sufficient statistic. 
Consider the $i^{\prime}$-th vote from $P^{\prime}$ and choose any resolvable pair $[i_{\rm im},i_{\rm pr}]$ of $I$.
Then by at most three swap operations among two votes of $I$,  we can 
1) increase the number of concurrences $C_{i_{\rm pr},i^{\prime}}$, 
or 2) make $I$ proper without changing the $i_{\rm pr}$-th vote of $I$. 
Furthermore, 
if the resulting dataset is improper then its improper vote and the $i_{\rm pr}$-th vote form a resolvable pair, 
and each intermediate swap operation between two consecutive 
improper datasets is compatible with them.
\end{lemma}

To prove Lemma \ref{lem:improper transposition lemma}, we need the following two lemmas.
For an improper dataset $I$, we denote $\bar{I} = \{ \ii _{ij} \}$.
\begin{lemma}
\label{lem:extended-1}
Let $I$ be an improper dataset with $\ii_{i_{\rm im} j}=b+c-a$.
Suppose that $\ii_{ij^{\prime}} = a$ where $i\neq i_{\rm im}, j^{\prime} \neq j$.
Then, letting $\ii_{ij}=d \neq a$, $I$ can be transformed by a swap operation among two votes $R=\{i_{\rm im},i\}$ to another improper dataset $I^{\prime}$ containing the improper $i_{\rm im}$-th vote where $\ii^{\prime}_{i_{\rm im} j^{\prime}} = a$ and $\ii^{\prime}_{i_{\rm im} j}$ is either of $b+c-a$, $b+d-a$ or $c+d-a$.
\end{lemma}

\begin{proof}
Without loss of generality, assume that $i_{\rm im}=j=1$ and $i=j^{\prime}=2$.
Let $\ii_{12}=e\neq b,c$.
We first make a swap of $\{ 1,2 \}: a \swapj{2} e$ to $I$:
\[
\left[
\begin{tabular}{ccccc}
$b+c-a$&$e$&$*$&$\cdots$&$*$\\
$d$&$a$&$$*$$&$\cdots$&$*$\\
$*$&$*$&$*$&$\cdots$&$*$\\
$\vdots$&$\vdots$&$\vdots$&$\vdots$&$\vdots$\\
$*$&$*$&$*$&$\cdots$&$*$\\
\end{tabular}
\right]
\rightarrow
\left[
\begin{tabular}{ccccc}
$b+c-a$&$a$&$*$&$\cdots$&$*$\\
$d$&$e$&$$*$$&$\cdots$&$*$\\
$*$&$*$&$*$&$\cdots$&$*$\\
$\vdots$&$\vdots$&$\vdots$&$\vdots$&$\vdots$\\
$*$&$*$&$*$&$\cdots$&$*$\\
\end{tabular}
\right].
\]
After the swap, $a$ may collide in the first vote and $e$ may collide in the second vote.
If there is no collision, the claim of this lemma is proved.
Otherwise, we can resolve these possible collisions in the following way.

We try to resolve the collision of $e$ in the second vote as in Lemma \ref{lem:collision lemma} considering a swap process:
\begin{equation} \label{eq:collision_second_row}
	\{1,2\}: e \swap s_1 \swap s_2 \swap \cdots  .
\end{equation}
In this process the collisions always occur in the second vote.

Consider the case that $d$ is equal to $b$ or $c$, say $d=b$.
Since the first and the second votes contain $b$ only in the first position,
$b$ does not collide in \eq{eq:collision_second_row}, which implies that $c$ does not collide in \eq{eq:collision_second_row}.
Since no $a$ is in the second vote at the beginning of \eq{eq:collision_second_row}, $a$ does not collide in \eq{eq:collision_second_row}.
Therefore, there is no swap involving the first position in \eq{eq:collision_second_row}, which implies that the collision of $e$ can be resolved as in Lemma \ref{lem:collision lemma}.

Consider the case $d\neq b,c$.
The difference of this case from Lemma \ref{lem:collision lemma} is that the process \eq{eq:collision_second_row} may hit the first position.  
This happens when $d$ appears in \eq{eq:collision_second_row} for the fist time as $s_l \swapj{j} d$, $j\neq 1$, and $d=\ii_{1j}$ in the first vote is swapped down to the second vote in the $j$-th position.  
Then we need to choose $b$ or $c$ and make the swap $d\swap b$ or $d\swap c$ in the first position.
By symmetry, without loss of generality, we perform $d \swapj{1} c$:
\[
\left[
\begin{tabular}{c}
$b+c-a$\\
$d$
\end{tabular}
\right]
\rightarrow
\left[
\begin{tabular}{c}
$b+d-a$\\
$c$
\end{tabular}
\right].
\]
This amounts to ignoring $b$ and $-a$ and we look at the improper element $b+c-a$ just as a proper element $c$ in resolving the collision of $e$. 
We leave $b-a$ in the $(1,1)$-element of $\bar{I}$ as it is during the sequence in \eq{eq:collision_second_row}.
Then just as in Lemma \ref{lem:collision lemma} it follows that no candidate appears twice in \eq{eq:collision_second_row}. 
Note that $b$ and $-a$ which were left in the $(1,1)$-element cause no trouble, because collision occurs always in the second vote. 
Indeed, $b$ causes no trouble because it does not leave the first vote.  
The candidate $a$ causes no trouble because the second vote does not initially contain $a$ and when $a$ is swapped from the first vote to the second vote, then the process in \eq{eq:collision_second_row} ends at that point.

After the collision of $e$ is resolved, $a$ may still collide in the first vote.
Let $j_1$ and $j_2$, $j_1 \neq j_2$, be the labels of positions containing $a$ in the first vote other than the first position.
To resolve this collision we consider the following two swap processes:
\begin{eqnarray} 
	\label{eq:process_one}
	&& \{1,2\}: a \swapj{j_1} s_1 \swap s_2 \swap \cdots , \\
	\label{eq:process_two}
	&& \{1,2\}: a \swapj{j_2} s_1^{\prime} \swap s_2^{\prime} \swap \cdots ,
\end{eqnarray}
where no swap in the $j_2$-th position is involved in \eq{eq:process_one} and no swap in the $j_1$-th position is involved in \eq{eq:process_two}.
Since every candidate in the first and second votes except $a$ appears in at most two positions,  the common candidate involved both in \eq{eq:process_one} and in \eq{eq:process_two} is $a$ only.
Then one of  \eq{eq:process_one} and \eq{eq:process_two}, say \eq{eq:process_one}, involves neither $b$ nor $c$, or involves $b$ and no $c$.
Therefore, ignoring $c,-a$ and $a$ in the $j_2$-th column, we see that the swap process \eq{eq:process_one} ends in finite number of steps as in Lemma \ref{lem:collision lemma}.
\end{proof}

\begin{lemma}
\label{lem:extended-2}
Let $I$ be an improper dataset with an improper element $\ii_{i_{\rm im} j}=b+c-a$.
Let $\ii_{i j}=d, i\neq i_{\rm im}$, and suppose that $d \neq a,b,c$.
Then $I$ can be transformed to another improper dataset $I^{\prime}$ by a swap operation
among two votes $R=\{i_{\rm im},i\}$ such that either $\ii^{\prime}_{i_{\rm im} j} =b+d-a, \ii^{\prime}_{ij}=c$ or $\ii^{\prime}_{i_{\rm im} j} =c+d-a, \ii^{\prime}_{ij}=b$.
\end{lemma}

\begin{proof}
Without loss of generality assume $i_{\rm im} = j =1$ and $i=2$. 
Then the upper-left $2\times 1$ submatrix of $I$ is $\twoelementcolumn{b+c-a}{d}$.
Note that $[1,2]$ is not a resolvable pair because $d\neq a$.

We begin by considering two swaps of $\{ 1,2 \}: d \swapj{1} b$ and $\{ 1,2 \}: d \swapj{1} c$.
If $\{ 1,2 \}: d \swapj{1} b$ is applied to $I$, $b$ may collide in the second vote and $d$ may collide in the first vote.
If $\{ 1,2 \}: d \swapj{1} c$ is applied to $I$, $c$ may collide in the second vote and $d$ may collide in the first vote.
Considering the resolution of possible collisions in the second vote for each swap, the following two swap processes are obtained:
\begin{eqnarray} 
	\label{eq:lem_pro_one}
	&& \{1,2\}: d \swapj{1} b \swap s_1 \swap s_2 \swap \cdots , \\
	\label{eq:lem_pro_two}
	&& \{1,2\}: d \swapj{1} c \swap s^{\prime}_1 \swap s^{\prime}_2 \swap \cdots .
\end{eqnarray}
Since the number of positions which contains $a$ in the first or second vote is at most three, one of \eq{eq:lem_pro_one} and \eq{eq:lem_pro_two}, say \eq{eq:lem_pro_one}, contains $a$ at most once.
Note that each candidates other than $a$ appears in the first and second votes at most twice.
If \eq{eq:lem_pro_one} does not contain $a$, we see that \eq{eq:lem_pro_one} ends in finite number of steps as in Lemma \ref{lem:collision lemma}.
If \eq{eq:lem_pro_one} contains one $a$, the finiteness of \eq{eq:lem_pro_one} is proved by applying the similar discussion of Lemma \ref{lem:collision lemma} for the subprocess of \eq{eq:lem_pro_one} which starts from $a$.

After resolving the collision of $b$ in the second vote, $d$ may still collide in the first vote.
At this point the second vote contains at most one $a$.
Consider a swap process
\begin{equation} \label{eq:extended_2_2}
	\{1,2\}: d \swap s^{\prime \prime}_1 \swap s^{\prime \prime}_2 \swap \cdots .
\end{equation}
Since $b$ has already been involved in \eq{eq:lem_pro_one}, no $s^{\prime \prime}_i$ is equal to $b$.
If some $s^{\prime \prime}_i$ is $c$, the chain swap
\[
	\{1,2\}: d \swap s^{\prime \prime}_1 \swap s^{\prime \prime}_2 \swap \cdots \swap c
\]
resolves the collisions in the first vote.
Since $a$ appears in the second vote at most once, the process \eq{eq:extended_2_2} contains $a$ at most once.
If $a$ does not appear in \eq{eq:extended_2_2}, the process does not hit the first position and we see that \eq{eq:lem_pro_one} ends in finite number of steps as in Lemma \ref{lem:collision lemma}.
If $a$ appears in \eq{eq:extended_2_2}, the finiteness ~~~~~~~~~~~~~~~~~~~~~~~~~~of \eq{eq:lem_pro_one} is proved by applying the similar discussion of Lemma \ref{lem:collision lemma} for the subprocess of \eq{eq:lem_pro_one} which starts from $a$.
\end{proof}

Using Lemmas \ref{lem:extended-1} and \ref{lem:extended-2} we shall prove Lemma \ref{lem:improper transposition lemma}.

\begin{proof}[Proof of Lemma \ref{lem:improper transposition lemma}]
Without loss of generality, let $i^{\prime}=1$, 
$[i_{\rm im},i_{\rm pr}]=[2,1]$,  $\ii_{11}=a$, and $\ii_{21}=b+c-a$.
Then $\bar{I}$ looks like
\[
\left[
\begin{tabular}{cccc}
$a$&$*$&$\cdots$&$*$\\
$b+c-a$&$*$&$\cdots$&$*$\\
$*$&$*$&$\cdots$&$*$\\
$\vdots$&$\vdots$&$\vdots$&$\vdots$\\
$*$&$*$&$\cdots$&$*$\\
\end{tabular}
\right] .
\]
In the cases below, where a resulting dataset is improper, $[2,1]$ will be a resolvable pair.
\begin{description}
\item[Case 1] $p^{\prime}_{11}=a$.\\
In this case in $P^{\prime}$ and hence in $I$, the candidate $a$ appears at least once in the first position. 
Therefore $a$ is in the first position in some vote $i>2$ in $I$.  
Let $i=3$. 
Then the votes $[2,3]$ of $I$ form a revolvable pair and $I$ can be transformed to a proper dataset by Lemma \ref{lem:getting proper lemma}.  
This corresponds ~~~~~~to 2) of the lemma and is summarized as $I \swapoperation{\{ 2,3 \} } P$.

\item[Case 2] $p^{\prime}_{11}\neq a$, but $a$ appears in the first vote of $P^{\prime}$.\\
Without loss of generality let $p^{\prime}_{12}=a$.  Let $d=\ii_{12}$.
\begin{description}
\item[Case 2-1]  $\ii_{22}=a$.\\
We perform the double swap $a\swapj{1} d \swapj{2} a$ to the first two votes
\[
\left[
\begin{tabular}{ccccc}
$a$&$d$&$*$&$\cdots$&$*$\\
$b+c-a$&$a$&$*$&$\cdots$&$*$\\
$*$&$*$&$*$&$\cdots$&$*$\\
$\vdots$&$\vdots$&$\vdots$&$\vdots$&$\vdots$\\
$*$&$*$&$*$&$\cdots$&$*$\\
\end{tabular}
\right]
\ \rightarrow \  
\left[
\begin{tabular}{ccccc}
$d$&$a$&$*$&$\cdots$&$*$\\
$b+c-d$&$d$&$*$&$\cdots$&$*$\\
$*$&$*$&$*$&$\cdots$&$*$\\
$\vdots$&$\vdots$&$\vdots$&$\vdots$&$\vdots$\\
$*$&$*$&$*$&$\cdots$&$*$
\end{tabular}
\right].
\]
This increases $C_{11}$. This corresponds to 1) of the lemma and is summarized as $I \swapoperation{\{ 1,2 \} } I$.

\item[Case 2-2]  $\ii_{22}\neq a$.\\
Since $p^{\prime}_{12}=a$, $a$ has to appear in the second position of $I$.
Without loss of generality, let $\ii_{32}=a$. Let $e=\ii_{31}$ and $f=\ii_{22}$. 
Then $\bar{P}$ looks like
\[
\left[
\begin{tabular}{ccccc}
$a$&$d$&$*$&$\cdots$&$*$\\
$b+c-a$&$f$&$*$&$\cdots$&$*$\\
$e$&$a$&$*$&$\cdots$&$*$\\
$*$&$*$&$*$&$\cdots$&$*$\\
$\vdots$&$\vdots$&$\vdots$&$\vdots$&$\vdots$\\
$*$&$*$&$*$&$\cdots$&$*$
\end{tabular}
\right].
\]
From Lemma \ref{lem:extended-1} applied to votes $\{2,3\}$, this case is reduced to Case 2-1.
This case together with the subsequent operation of Case 2-1 is summarized as 
$I \swapoperation{\{ 2,3 \} } I \swapoperation{\{ 1,2 \} } I$.

\end{description}

\item[Case 3] $a$ does not appear in the first vote of $P^{\prime}$.\\
Let $d=p^{\prime}_{11}$, $d\neq a$.  
If $d=b$ or $d=c$, we directly go to the Cases 3-1 or Case 3-2 below.
If $d\neq b,c$, we need an extra step as follows.
Let $\ii_{31}=d$ without loss of generality. 
Then $\bar{I}$ looks like
\[
\left[
\begin{tabular}{cccc}
$a$&$*$&$\cdots$&$*$\\
$b+c-a$&$*$&$\cdots$&$*$\\
$d$&$*$&$\cdots$&$*$\\
$*$&$*$&$\cdots$&$*$\\
$\vdots$&$\vdots$&$\vdots$&$\vdots$\\
$*$&$*$&$\cdots$&$*$
\end{tabular}
\right] .
\]

By Lemma \ref{lem:extended-2} applied to votes $\{2,3\}$, we move $d$ to the second vote resolving the possible collisions.
At this point the $(2,1)$-element  of $I$ may be $b+d-a$ or $c+d-a$.  
We consider the former case without loss of generality.
Then $\bar{I}$ looks like
\begin{equation}
\label{eq:case-3-after-extended-2}
\left[
\begin{tabular}{cccc}
$a$&$*$&$\cdots$&$*$\\
$b+d-a$&$*$&$\cdots$&$*$\\
$c$&$*$&$\cdots$&$*$\\
$*$&$*$&$\cdots$&$*$\\
$\vdots$&$\vdots$&$\vdots$&$\vdots$\\
$*$&$*$&$\cdots$&$*$
\end{tabular}
\right] .
\end{equation}
\begin{description}
\item[Case 3-1] $d$ appears in the first vote of $I$.\\
Let $d=\ii_{12}$ without loss of generality.
We apply a double swap $a\swapj{1}d\swapj{2}a$:
\[
\left[
\begin{tabular}{ccccc}
$a$&$d$&$*$&$\cdots$&$*$\\
$b+d-a$&$e$&$*$&$\cdots$&$*$\\
$*$&$*$&$*$&$\cdots$&$*$\\
$\vdots$&$\vdots$&$\vdots$&$\vdots$&$\vdots$\\
$*$&$*$&$*$&$\cdots$&$*$\\
\end{tabular}
\right]
\rightarrow
\left[
\begin{tabular}{ccccc}
$d$&$a$&$*$&$\cdots$&$*$\\
$b$&$d+e-a$&$*$&$\cdots$&$*$\\
$*$&$*$&$*$&$\cdots$&$*$\\
$\vdots$&$\vdots$&$\vdots$&$\vdots$&$\vdots$\\
$*$&$*$&$*$&$\cdots$&$*$\\
\end{tabular}
\right]. 
\]
This case is summarized as $I\swapoperation{\{1,2\}}I$ or 
$I\swapoperation{\{2,3\}} I\swapoperation{\{1,2\}}I$, where $I\swapoperation{\{2,3\}}I$
is needed for the case $d\neq b,c$.  
We do not repeat this comment for the other cases below.

\item[Case 3-2] $d$ does not appear in the first vote of $I$.\\
If $a$ appears only once in the second vote, say in the $j$-th position, $j>1$, then we can apply the swap
$\{1,2\}: a\swapj{1}d$ to make $I$ proper, which is summarized as
$I\swapoperation{\{1,2\}}P$ or 
$I\swapoperation{\{2,3\}} I\swapoperation{\{1,2\}}P$. 

Hence we consider the case that  $a$ appears in two positions labeled by $j_1, j_2$, $1 < j_1 < j_2$ of the second vote of $I$.
Since $P^{\prime}$ does not contain $a$ in the first vote, $I$ has a vote not ~~~~~~~containing $a$.

\begin{description}
\item[Case 3-2-1] The third vote of \eqref{eq:case-3-after-extended-2} contains $a$.\\
Without loss of generality, suppose  the fourth vote of $I$  does not contain $a$.
Denote $e=\ii_{41}$. 
Interpreting two $a$'s in the second vote as a collision, we try to resolve the collision by swapping $a$ down to the fourth vote.
Then we have two processes of swaps
\begin{eqnarray} 
	\label{eq:without_d}
	&& \{2,4\} : a\swapj{j_1} s_1 \swap s_2 \swap \cdots , \\
	\label{eq:with_d}
	&& \{2,4\} : a\swapj{j_2} s^{\prime}_1 \swap s^{\prime}_2 \swap \cdots .
\end{eqnarray}
During these processes the collisions occur in the second vote.
Only one of these two can contain $d$.
Then we can choose a process, say \eq{eq:without_d}, which does not contain $d$.
As in Lemma \ref{lem:collision lemma} no candidate appears twice in \eq{eq:without_d}.
Hence, \eq{eq:without_d} is a finite chain swap resolving the collisions.

At this stage $\bar{I}$ looks like
\[
\left[
\begin{tabular}{cccc}
$a$&$*$&$\cdots$&$*$\\
$b+d-a$&$*$&$\cdots$&$*$\\
$c$&$*$&$\cdots$&$*$\\
$e$&$*$&$\cdots$&$*$\\
$*$&$*$&$\cdots$&$*$\\
$\vdots$&$\vdots$&$\vdots$&$\vdots$\\
$*$&$*$&$\cdots$&$*$
\end{tabular}
\right]
\ \ \mathrm{or} \  \ 
\left[
\begin{tabular}{ccccc}
$a$&$*$&$\cdots$&$*$\\
$e+d-a$&$*$&$\cdots$&$*$\\
$c$&$*$&$\cdots$&$*$\\
$b$&$*$&$\cdots$&$*$\\
$*$&$*$&$\cdots$&$*$\\
$\vdots$&$\vdots$&$\vdots$&$\vdots$\\
$*$&$*$&$\cdots$&$*$
\end{tabular}
\right].
\]
In either case, the swap $\{1,2\}: a\swapj{1} d$ increases $C$ and makes $I$ proper.
The whole process for this case is summarized as 
$I \swapoperation{\{2,4\}} I \swapoperation{\{1,2\}} P$
or 
$I\swapoperation{\{2,3\}}I \swapoperation{\{2,4\}} I \swapoperation{\{1,2\}} P$.

\item[Case 3-2-2] The third vote of \eqref{eq:case-3-after-extended-2} does not contain $a$.\\
We can just use the third vote of \eqref{eq:case-3-after-extended-2} as the fourth vote of the previous case. Hence 
$I \swapoperation{\{2,4\}} I$ is replaced by $I \swapoperation{\{2,3\}} I$ and 
this case is summarized as
$I\swapoperation{\{2,3\}}I \swapoperation{\{1,2\}} P$.
\end{description}
\end{description}
\end{description}
\end{proof}

We now summarize what we have proved so far.  We will again discuss the following result in Section \ref{subsec:negative}.

Let $P$ and $P^{\prime}$ be two proper datasets with the same sufficient statistic, respectively.  
Suppose that the $i$-th vote of $P$ and the $i^{\prime}$-th vote of $P^{\prime}$ are different, i.e., $V^{(i)} \neq V^{(i^{\prime})}$.
If we allow improper datasets, then by a sequence of swap operations among two votes of $P$, we can make the $i$-th vote of $P$ identical with the $i^{\prime}$-th vote of $P^{\prime}$.  
Then we throw away this common vote from the two datasets and repeat the procedure.  
It should be noted that $P$ may have been transformed to an improper dataset $I$ when two votes coincide, 
but $I$ contains a resolvable pair $[i_{\rm im},i_{\rm pr}]$ with $i_{\rm pr}\neq i$.
Hence we can continue this process until $P$ is fully transformed to $P^{\prime}$.

In order to finish our proof of Theorem \ref{thm:main}, we have to show that each intermediate improper dataset can be temporarily transformed to a proper dataset and the consecutive proper datasets are connected by operations among three votes.   

We decompose the whole process of transforming $P$ to $P^{\prime}$ into segments that consist of transformations from a proper dataset to another proper dataset with improper intermediate steps.
One segment is depicted as follows:
\begin{equation}
P_1 \longleftrightarrow I_1
 \longleftrightarrow \cdots
 \longleftrightarrow  I_i
 \longleftrightarrow  I_{i+1}
 \longleftrightarrow \cdots
 \longleftrightarrow  I_m 
 \longleftrightarrow P_m,
\label{eq:all move}
\end{equation}
where each $\longleftrightarrow$ (omitting $R$) denotes a swap operation among two votes in Lemmas \ref{lem:proper transposition lemma} and \ref{lem:improper transposition lemma}. 
By these lemmas, the number of concurrences in $P_m$ is larger than in $P_1$.
We claim that for any consecutive improper datasets $I_i$, $I_{i+1}$, we can find proper datasets $P_i,P^{\prime}_i,P^{\prime}_{i+1}$ satisfying
\begin{eqnarray}
&&P_i
 \longleftrightarrow  I_i
 \longleftrightarrow  I_{i+1}
 \longleftrightarrow P^{\prime}_{i+1}\label{eq:partial move}, \\
&&P^{\prime}_i
 \longleftrightarrow  I_i
 \longleftrightarrow P_i\label{eq:partial move 2}.
\end{eqnarray}
The swap operation for $I_i \longleftrightarrow  I_{i+1}$ is compatible with both datasets.
Hence if we choose a common resolvable pair for $I_i$ and $I_{i+1}$, then 
$\eqref{eq:partial move}$ for transforming $P_i$ to $P^{\prime}_i$ involves three votes.
On the other hand, since both $P^{\prime}_i \longleftrightarrow I_i$ and $I_i \longleftrightarrow P_i$ involve an improper vote, the operation of transforming $P^{\prime}_i$ to $P_i$ involves three votes.
This completes the proof of Theorem \ref{thm:main}.

\section{Structure of moves of degree two and three}
\label{sec:move-structure}

To analyze the structure of moves in the Markov basis,
it is enough to consider the moves of degree two and three because of Theorem \ref{thm:main}.
It means that we only need to consider the dataset consisting of two or three votes.
Then we can analyze the structure of moves by discussing the structure of fiber for the sufficient statistic.
Details of computational results used in this section are available at \cite{source}.

\subsection{Moves of degree two}
\label{sec:structure-two}

We begin by discussing the structure of fibers for datasets which consist of two votes.
Consider a sequence of multisets consisting of two elements in $[n]$ of the form:
\begin{equation}
\label{eq:sufficient-statistic-for-tow-rows}
	M=(\{a_1, a_2\}, \ldots, \{ a_{2r-1}, a_{2r} \} ),
      \end{equation}
where $a_j \in [n], j=1,\ldots,2r$.
Each multiset $\{ a_{2j-1}, a_{2j} \}$ 
corresponds to the multiset of two candidates in the $j$-th position.
This sequence is a possible observation of sufficient statistic in the $(n,r)$-Birkhoff model
if and only if each $k \in [n]$ appears in the multiset $\{ a_1,a_2,\ldots, a_{2r} \}$ at most twice.
For the observation $( \{a_1, a_2\}, \ldots, \{ a_{2r-1}, a_{2r} \} )$ of the sufficient statistic,
define a graph $G_M$ on the vertex set $[r]$ as follows:
for each $j,j^{\prime} \in [r], j \neq j^{\prime}$, an edge $\{ j, j^{\prime} \}$ of $G_M$ exists 
if and only if $\{ a_{2j-1}, a_{2j} \} \cap \{ a_{2j^{\prime}-1}, a_{2j^{\prime}} \} \neq \emptyset$.
We call the multiset $\{ a_{2j-1}, a_{2j} \}$ the $j$-th block for $j \in [r]$.

For an isolated vertex in $G_M$ the corresponding block 
has the form either $\{ k, k \}$ or $\{ k, k^{\prime}\}$ for some $k, k^{\prime} \in [n], k \neq k^{\prime}$.
In the former case there is no necessity to distinguish two votes by this block.
In the latter case the votes might be distinguished by this block.
Since every non-isolated vertex is contained by at most two edges,
each connected components of $G_M$ is a chain or a cycle if it consists of more than one vertex.
Then the candidates are uniquely assigned to two votes as a subset of the votes.
Let $L$ be the number of connected components of $G_M$ 
brushing aside those of the form $\{ k, k\}$ for some $k \in [n]$.
The number of elements of the corresponding fiber is $2^{L-1}$.
Especially, the move arising from the corresponding fiber is indispensable
if and only if $L=2$.

We evaluate the number of moves of degree two in a minimal Markov basis.
For the case $r=2$ the number of moves of degree two is $6\binom{n}{4}$.
For the case $r=3$, the number of moves of degree two is
\begin{eqnarray*}
	&&1 \binom{n}{6} \frac{6!}{2!2!2!} (4-1) + \binom{3}{2} \binom{n}{5} \frac{5!}{2!2!1!} (2-1) + \binom{3}{2} \binom{n}{5} \frac{5!}{2!1!1!1!} (2-1) + \binom{3}{2} \binom{n}{4} \frac{4!}{2!2!} (2-1) \\
	&=& 18 \binom{n}{4} + 270 \binom{n}{5} + 270 \binom{n}{6}.
\end{eqnarray*}

Consider the case $r=4$. 
The number of moves of degree two for the partition $(3,1)$ of four is 
\begin{eqnarray*}
	\binom{4}{3} \binom{n}{6} \frac{6!}{1!1!1!1!2!} (2-1) + \binom{4}{3} \binom{n}{5} \frac{5!}{1!1!1!2!} (2-1) 
	= 240 \binom{n}{5} + 1440 \binom{n}{6}. 
\end{eqnarray*}
The number of moves of degree two for the partition $(2,2)$ of four is 
\begin{eqnarray*}
	&& \frac{1}{2!} \binom{4}{2} \binom{n}{6} \frac{6!}{1!1!1!1!1!1!}(2-1) + \binom{4}{2} \binom{n}{5} \frac{5!}{1!1!1!2!}(2-1) + 
	\frac{1}{2!} \binom{4}{2} \binom{n}{4} \frac{4!}{2!2!}(2-1) \\
	&=& 18 \binom{n}{4} + 360 \binom{n}{5} + 2160 \binom{n}{6}.
\end{eqnarray*}
The number of moves of degree two for the partition $(2,1,1)$ of four is
\begin{eqnarray*}
	&& \binom{4}{2} \binom{n}{7} \frac{7!}{1!1!1!2!2!} (4-1) + \frac{4!}{2!1!1!} \binom{n}{6} \frac{6!}{1!1!1!2!1!} (2-1) + \binom{4}{2} \binom{n}{6} \frac{6!}{2!2!2!} (4-1) \\
		&& + \frac{4!}{2!1!1!} \binom{n}{5} \frac{5!}{2!2!1!} (2-1) \\
	&=& 360 \binom{n}{5} + 5940 \binom{n}{6} + 22680 \binom{n}{7}.
\end{eqnarray*}
The number of moves of degree two for the partition $(1,1,1,1)$ of four is 
\begin{eqnarray*}
	&& n_8\frac{8!}{2!2!2!2!}(8-1) + \binom{4}{3} \binom{n}{7} \frac{7!}{2!2!2!1!}(4-1) + \binom{4}{2} \binom{n}{6} \frac{6!}{2!2!1!1!}(2-1) \\
	&=& 1080 \binom{n}{6} + 7560 \binom{n}{7} + 17640 \binom{n}{8}.
\end{eqnarray*}
Then the number of moves of degree two for $r=4$ is
\begin{eqnarray*}
	18 \binom{n}{4} + 960 \binom{n}{5} + 10620 \binom{n}{6} + 30240 \binom{n}{7} + 17640 \binom{n}{8}.
\end{eqnarray*}

By the similar calculation we obtain the following polynomial which represents the number of moves of degree two for the case $r=5$:
\[
	1050 \binom{n}{5} + 40050 \binom{n}{6} + 485100 \binom{n}{7} + 2444400 \binom{n}{8} + 3969000 \binom{n}{9} + 1701000 \binom{n}{10}.
\]

The number of moves of degree two in minimal Markov bases are summarized as \tabref{table:number_of_two}.
The authors confirmed that the numbers above the horizontal lines in \tabref{table:number_of_two} coincide with the numbers obtained by the software 4ti2\cite{4ti2}.

\begin{table}[t]
\centering
\renewcommand{\tabcolsep}{5pt}
{\footnotesize
\caption{Number of moves of degree two.}
\label{table:number_of_two}
\begin{tabular}{|r||r|r|r|r|}
\hline
\backslashbox{$n$}{$r$} &2&3&4&5\\\hline\hline
1&0&0&0&0\\
2&0&0&0&0\\
3&0&0&0&0\\
4&6&18&18&0\\
5&30&360&1050&1050\\\cline{4-5}
6&90&2160&16650&46350\\
7&210&8190&125370&787500\\
8&420&23940&611940&7505400\\\cline{3-3}
9&756&58968&2262708&46928700\\
10&1260&128520&6898500&218276100\\\hline
\end{tabular}
}
\end{table}

\subsection{Moves of degree three}
\label{subsec:degree three}

The structure of fibers for datasets which consists of three votes is more complicated.
Let 
\[
	M = ( \{a_1, b_1, c_1 \}, \ldots, \{ a_r, b_r, c_r \} )
\]
be the observed sufficient statistic 
where $a_j,b_j,c_j \in [n], j=1,\ldots,r$ and each $k \in [n]$ appears in the multiset
$\{ a_1,\ldots,a_r, b_1,\ldots, b_r, c_1,\ldots, c_r \}$ at most three times. 
Similarly to the case of two votes, a graph $G_M$ on $[r]$ can be defined:
an edge $\{ j, j^{\prime}\}$ of $G_M$ exists if and only if $\{ a_j, b_j, c_j \} \cap \{ a_{j^{\prime}}, b_{j^{\prime}}, c_{j^{\prime}} \} \neq \emptyset$.
For example, consider the case $n=6, r=3$ and let the set of candidates be $\{ a,b,c,d,e,f \}$.
Let 
\[
	M=(\{ a,a,b \}, \{ c,c,d\}, \{ d,e,f \} )
\]
be the observed sufficient statistic.
The vertex set of $G_M$ is $[3]$ and the connected components are $\{ 1 \}$ and $\{ 2,3 \}$.
The possible assignment in the first connected component is $\{ (a), (a), (b) \}$.
In the second connected component there are two kinds of assignments, $\{ (c,d) , (c,e), (d,f) \}$ and $\{ (c,d), (c,f), (d,e) \}$.
In this case the number of elements of the corresponding fiber is six.

Now we discuss the detailed structure of fibers arising from the sufficient statistic $M=( \{a_1, b_1, c_1 \}, $ $\ldots, \{ a_r, b_r, c_r \} )$ such that the associated graph $G_M$ is connected.
Thanks to the symmetry in permutation of ranking orders and of labels of the candidates, 
we consider the equivalence classes of such sufficient statistics.
Figures \ref{fig:2_need3_0}--\ref{fig:3_need3_5} show the graph $G_M$'s for all the representatives of the equivalence classes for $r=2,3$ whose corresponding fiber needs a move of degree three for its connectivity.
The moves of degree three arising from these figures except \figref{fig:3_need3_4} are indispensable.
On the other hand, to guarantee the connectivity of the fiber associated with \figref{fig:3_need3_4}, the Markov basis needs to include a dispensable move of degree three.
\begin{figure}[t]
\centering
\begin{minipage}{0.24\hsize}
\begin{center}
\includegraphics[width=2.4cm]{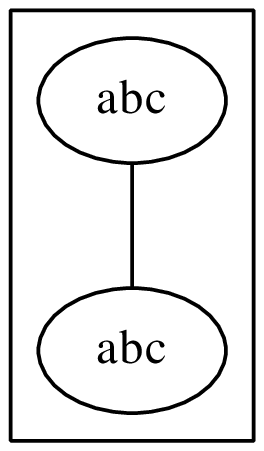}
\vspace{-5mm}
\figcaption{}
\label{fig:2_need3_0}
\end{center}
\end{minipage}
\begin{minipage}{0.24\hsize}
\begin{center}
\includegraphics[width=2.8cm]{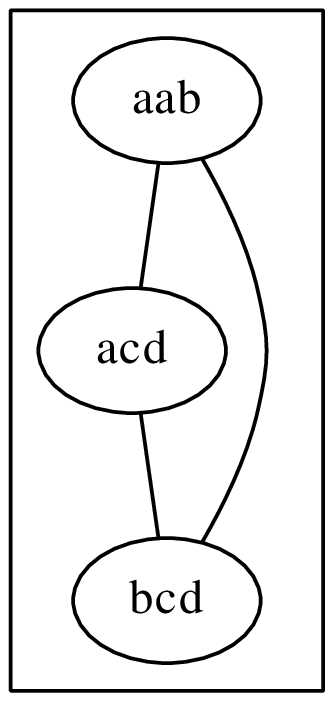}
\vspace{-5mm}
\figcaption{}
\label{fig:3_need3_0}
\end{center}
\end{minipage}
\begin{minipage}{0.24\hsize}
\begin{center}
\includegraphics[width=2.8cm]{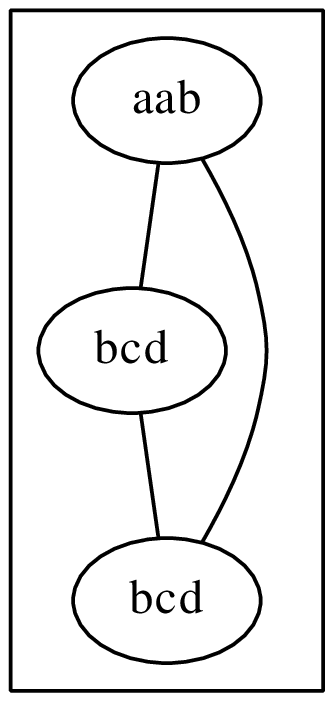}
\vspace{-5mm}
\figcaption{}
\label{fig:3_need3_1}
\end{center}
\end{minipage}
\begin{minipage}{0.24\hsize}
\begin{center}
\includegraphics[width=2.8cm]{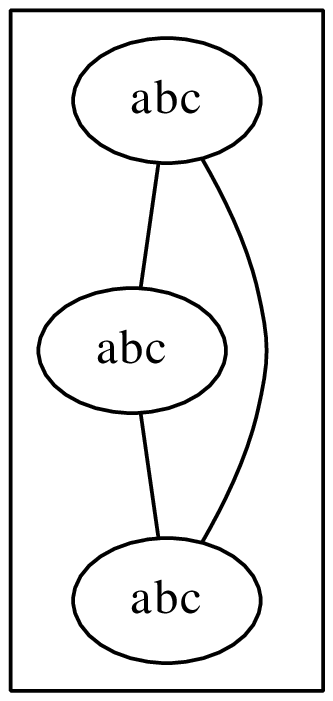}
\vspace{-5mm}
\figcaption{}
\label{fig:3_need3_2}
\end{center}
\end{minipage}
\end{figure}

\begin{figure}[t]
\centering
\begin{minipage}{0.24\hsize}
\begin{center}
\includegraphics[width=2.8cm]{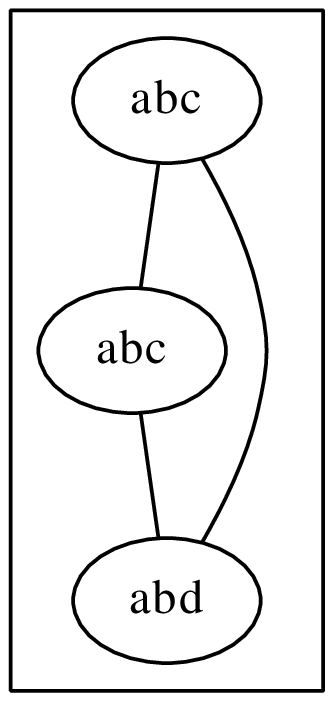}
\vspace{-5mm}
\figcaption{}
\label{fig:3_need3_3}
\end{center}
\end{minipage}
\begin{minipage}{0.24\hsize}
\begin{center}
\includegraphics[width=2.8cm]{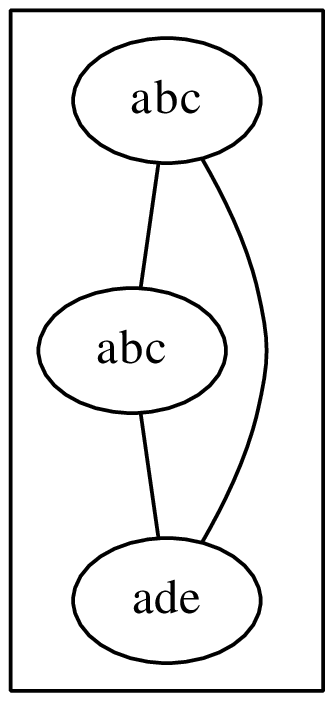}
\vspace{-5mm}
\figcaption{}
\label{fig:3_need3_4}
\end{center}
\end{minipage}
\begin{minipage}{0.24\hsize}
\begin{center}
\includegraphics[width=2.8cm]{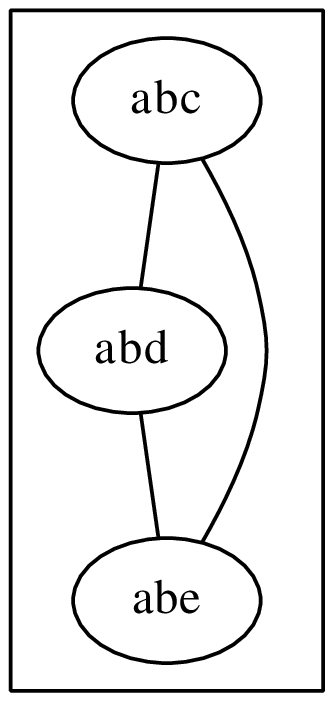}
\vspace{-5mm}
\figcaption{}
\label{fig:3_need3_5}
\end{center}
\end{minipage}
\end{figure}

Let us consider the case $r=4$.
There are $241$ different equivalence classes of the sufficient statistic $M$'s with connected $G_M$.
For $38$ classes among them, the corresponding fibers need moves of degree three for their connectivity.
\tabref{table:summary_four} summarizes the structure of the equivalence classes.
In this table, 38 equivalence classes are classified by the associated graph $G_M$, the number $n_M$ of candidates appearing in $M$, and whether a move of degree three needed for the connectivity of the corresponding fiber is indispensable or not.
\figref{fig:4_need3_0} shows an example of a fiber connected by an indispensable move of degree three.
On the other hand, the fiber in \figref{fig:4_need3_2} needs a dispensable move of degree three for its connectivity.
\newlength{\myheight}
\setlength{\myheight}{0.4cm}
\begin{table}[t]
\begin{center}
\renewcommand{\tabcolsep}{5pt}
{\footnotesize
\tblcaption{Classification of the equivalence classes for $r=4$.}
\label{table:summary_four}
\begin{tabular}{|c|r|c||r|}
\hline
$G_M$ & $n_M$ & indispensability & $\#$ of equiv. classes\\
\hline\hline
\setlength\unitlength{2mm}
\begin{picture}(4.5,4.5)(0,0)
	\put(0,0){\line(1,0){4}}
	\put(4,0){\line(0,1){4}}
	\put(0,0){\line(0,1){4}}
	\put(0,4){\line(1,0){4}}
	\put(0,0){\circle*{0.7}}
	\put(0,4){\circle*{0.7}}
	\put(4,0){\circle*{0.7}}
	\put(4,4){\circle*{0.7}}
\end{picture}
 & 5 & yes &2 \\ \hline
\parbox[c][\myheight][b]{0cm}{
\setlength\unitlength{2mm}
\begin{picture}(4.5,4.5)(0,0)
	\put(-2,-4.5){\line(1,0){4}}
	\put(2,-4.5){\line(0,1){4}}
	\put(-2,-4.5){\line(0,1){4}}
	\put(-2,-0.5){\line(1,-1){4}}
	\put(-2,-4.5){\circle*{0.7}}
	\put(-2,-0.5){\circle*{0.7}}
	\put(2,-4.5){\circle*{0.7}}
	\put(2,-0.5){\circle*{0.7}}
\end{picture}
}
 & 5 & yes & 2 \\ \cline{3-4}
 & & no & 2\\ \cline{2-4}
 & 6 & yes & 1\\ \cline{3-4}
 & & no & 6 \\ \cline{2-4}
 & 7 & no & 2 \\ \hline
\setlength\unitlength{2mm}
\parbox[c][\myheight][b]{0cm}{
\begin{picture}(4.5,4.5)(0,0)
	\put(-2,-4){\line(1,0){4}}
	\put(2,-4){\line(0,1){4}}
	\put(-2,-4){\line(0,1){4}}
	\put(-2,0){\line(1,0){4}}
	\put(-2,-4){\line(1,1){4}}
	\put(-2,-4){\circle*{0.7}}
	\put(-2,0){\circle*{0.7}}
	\put(2,-4){\circle*{0.7}}
	\put(2,0){\circle*{0.7}}
\end{picture}
}
 & 5 & yes & 7 \\ \cline{3-4}
 & & no & 1 \\ \cline{2-4}
 & 6 & no & 5 \\ \cline{2-4}
 & 7 & no & 1 \\ \hline
\setlength\unitlength{2mm}
\parbox[c][\myheight][b]{0cm}{
\begin{picture}(4.5,4.5)(0,0)
	\put(-2,-4.5){\line(1,0){4}}
	\put(2,-4.5){\line(0,1){4}}
	\put(-2,-4.5){\line(0,1){4}}
	\put(-2,-0.5){\line(1,0){4}}
	\put(-2,-4.5){\line(1,1){4}}
	\put(-2,-0.5){\line(1,-1){4}}
	\put(-2,-4.5){\circle*{0.7}}
	\put(-2,-0.5){\circle*{0.7}}
	\put(2,-4.5){\circle*{0.7}}
	\put(2,-0.5){\circle*{0.7}}
\end{picture}
}
 & 4 & yes & 1 \\ \cline{2-4}
 & 5 & yes & 2 \\ \cline{3-4}
 & & no & 4 \\ \cline{2-4}
 & 6 & yes & 1\\ \cline{3-4}
 & & no & 1 \\ \hline
\end{tabular}
}
\end{center}
\end{table}
\begin{figure}[t]
\begin{minipage}{0.50\hsize}
\begin{center}
\includegraphics[width=5.6cm]{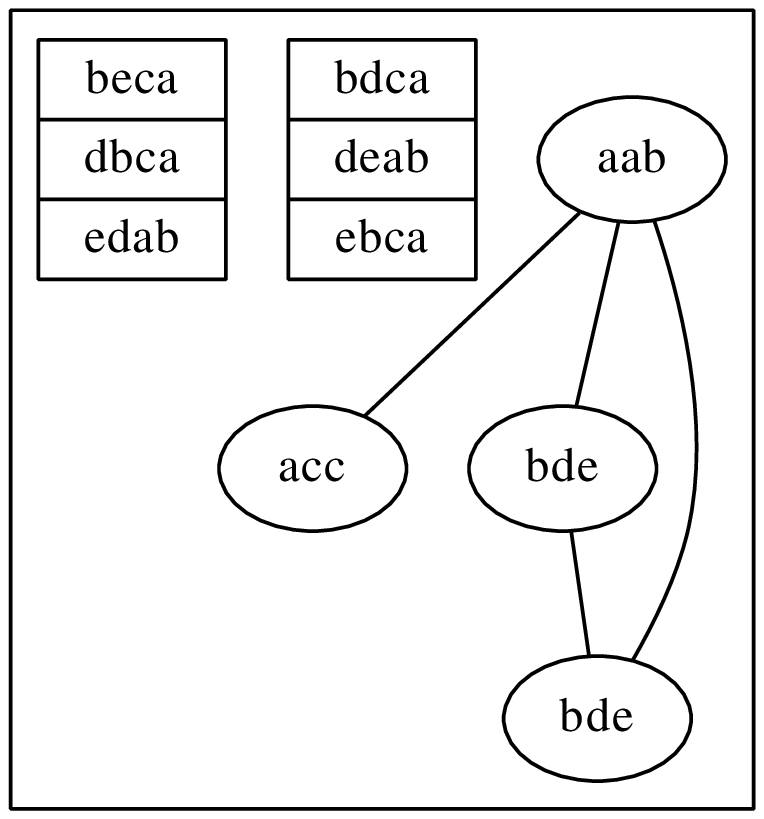}
\vspace{-5mm}
\figcaption{}
\label{fig:4_need3_0}
\end{center}
\end{minipage}
\begin{minipage}{0.50\hsize}
\begin{center}
\includegraphics[width=5.6cm]{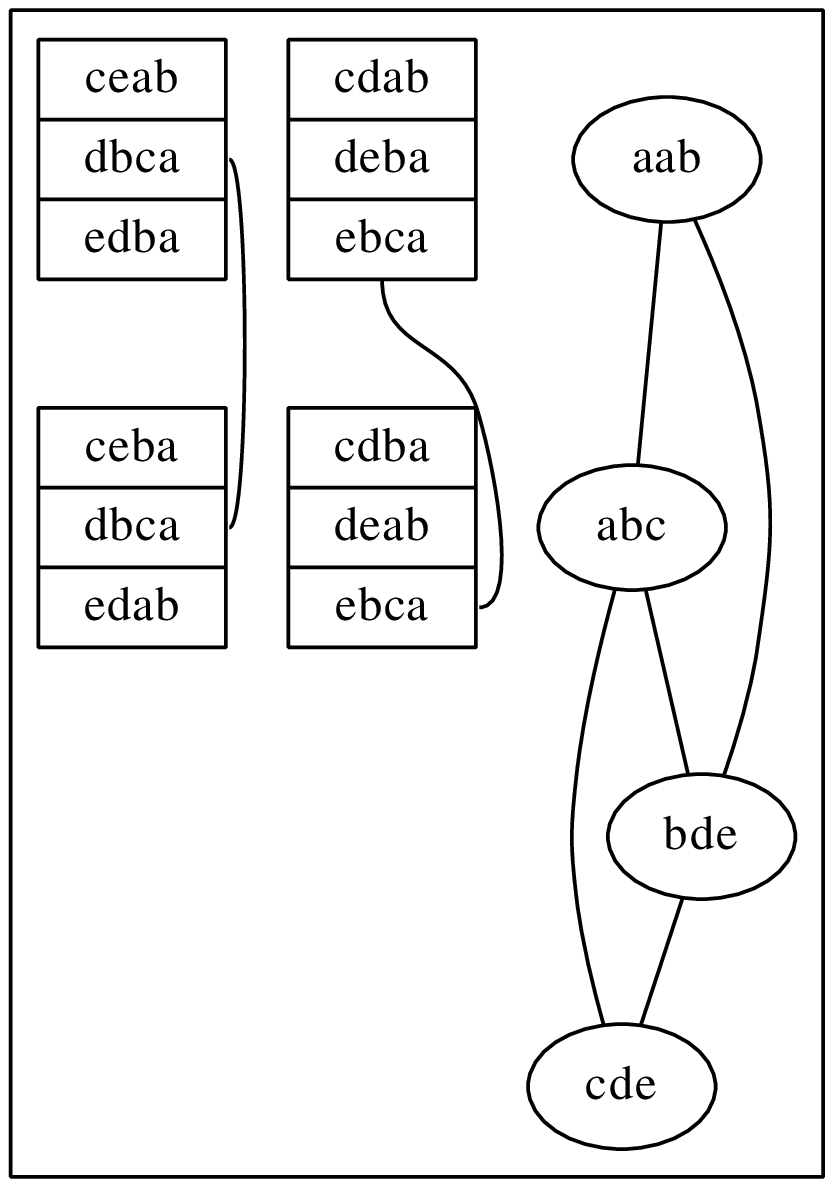}
\vspace{-5mm}
\figcaption{}
\label{fig:4_need3_2}
\end{center}
\end{minipage}
\end{figure}

The rest of this section is devoted to the evaluation of the number of moves of degree three in a minimal Markov basis.
We first evaluate the sizes of equivalence classes of the sufficient statistic $M$'s with connected $G_M$.
For the fiber $\mathcal{F}_M$ associated with a given sufficient statistic $M$, let $\mathcal{G}_{\mathcal{F}_M}$ be a graph on the vertex set $\mathcal{F}_M$ defined as follows: an edge $\{ x, y\}$ for $x,y\in \mathcal{F}_M$ exists if and only if $x$ and $y$ are connected by a move of degree two.
\tabref{table:number_of_graphs} counts the sufficient statistic $M$'s in each equivalence classes classified by the length $r$ of ranking, the number $n_M$ of the candidates appearing in $M$, and the number $N_M$ of connected components of the graph $\mathcal{G}_{\mathcal{F}_M}$.
\begin{table}[t]
\begin{center}
\renewcommand{\tabcolsep}{5pt}
{\footnotesize
\tblcaption{Sizes of equivalence classes of sufficient statistics.}
\label{table:number_of_graphs}
\begin{tabular}{|r|r|r||r| |r|r|r||r|}
\hline
$r$ & $n_M$ & $N_M$ & size of equiv. class & $r$ & $n_M$ & $N_M$ & size of equiv. class\\
\hline\hline
1&1&1&1 & 3&6&1&8820 \\\cline{2-4}\cline{6-8}
&2&1&2 & &7&1&4410 \\\cline{2-4}\cline{5-8}
&3&1&1 & 4&4&1&1128 \\\cline{1-4}\cline{7-8}
2&2&1&2 & &&2&144 \\\cline{2-4}\cline{6-8}
&3&1&30 & &5&1&82080 \\\cline{3-4}\cline{7-8}
&&2&1 & &&2&23040 \\\cline{2-4}\cline{7-8}
&4&1&60 & &&4&600 \\\cline{2-4}\cline{6-8}
&5&1&30 & &6&1&885240 \\\cline{1-4}\cline{7-8}
3&3&1&30 & &&2&60480 \\\cline{3-4}\cline{6-8}
&&2&1 & &7&1&2847600 \\\cline{2-4}\cline{7-8}
&4&1&1128 & &&2&37800 \\\cline{3-4}\cline{6-8}
&&2&144 & &8&1&3749760 \\\cline{2-4}\cline{6-8}
&5&1&5760 & &9&1&1814400 \\\cline{3-8}
&&2&150 & &&&\\\hline
\end{tabular}
}
\end{center}
\end{table}

Using \tabref{table:number_of_graphs}, the number of moves of degree three in a minimal Markov basis can be calculated.
To illustrate the process of this calculation we define some notations.
Let $[M]_{r^{\prime},n^{\prime},N^{\prime}}$ be the equivalence class whose length of ranking is $r^{\prime}$, the number of candidates is $n^{\prime}$, and the number of connected component of $\mathcal{G}_{\mathcal{F}_M}$ is $N^{\prime}$.
Let $n_{r^{\prime},n^{\prime},N^{\prime}} = n_M, M \in [M]_{r^{\prime},n^{\prime},N^{\prime}}$.
Let $N_{r^{\prime},n^{\prime},N^{\prime}} = N_M, M \in [M]_{r^{\prime},n^{\prime},N^{\prime}}$.
For example, $n_{2,3,1}=3$ and $N_{2,3,1}=1$.
Denote the size of equivalence class $[M]_{r,n^{\prime},N^{\prime}}$ by $\# [r,n^{\prime},N^{\prime}]$.
For example, $\# [1,1,1]=1$ and $\# [2,4,1]=60$.
Furthermore, for simplicity, we set
\[
	I[n, r^{\prime}, n^{\prime}, N^{\prime}] 
	= \binom{n}{n^{\prime}} \# [r^{\prime},n^{\prime},N^{\prime}].
\]

Consider the case $r=2$.
The number of moves of degree three is
\begin{eqnarray*}
	I[n,2,3,2] (N_{2,3,2}-1) = \binom{n}{3} \times 1 \times (2-1) = \binom{n}{3}.
\end{eqnarray*}

Consider the case $r=3$.
The number of moves of degree three for the partition $(3)$ of three is 
\begin{eqnarray*}
	&& I[n,3,3,2] (N_{3,3,2}-1) + I[n,3,4,2] (N_{3,4,2}-1) + I[n,3,5,2] (N_{3,5,2}-1) \\
	&=& \binom{n}{3} + 144 \binom{n}{4} + 150 \binom{n}{5}.
\end{eqnarray*}
The number of moves of degree three for the partition $(2,1)$ of three is 
\begin{eqnarray*}
	&& \binom{3}{2} I[n,2,3,2] 
		\times \Bigl( 
		I[n-n_{2,3,2}, 1,1,1] (N_{2,3,2} N_{1,1,1} -1) \\
		&& +I[n-n_{2,3,2}, 1,2,1] (N_{2,3,2} N_{1,2,1} -1) 
		+I[n-n_{2,3,2}, 1,3,1] (N_{2,3,2} N_{1,3,1} -1) 
		\Bigr) \\
	&=& 12 \binom{n}{4} + 60 \binom{n}{5} + 60 \binom{n}{6}.
\end{eqnarray*}
Then the number of moves of degree three for $r=3$ is
\begin{eqnarray*}
	\binom{n}{3} +156 \binom{n}{4} +210 \binom{n}{5} + 60 \binom{n}{6}.
\end{eqnarray*}

Consider the case $r=4$.
The number of moves of degree three for the partition $(4)$ of four is 
\begin{eqnarray*}
	&& I[n,4,4,2] (N_{4,4,2}-1) 
		+ I[n,4,5,2] (N_{4,5,2}-1) 
		+ I[n,4,5,4] (N_{4,5,4}-1) \\
		&& + I[n,4,6,2] (N_{4,6,2}-1) 
		+ I[n,4,7,2] (N_{4,7,2}-1) \\
	&=& 144 \binom{n}{4} + 24840 \binom{n}{5} + 60480 \binom{n}{6} +37800 \binom{n}{7}.
\end{eqnarray*}
The number of moves of degree three for the partition $(3,1)$ of four is 
\begin{eqnarray*}
	&& \binom{4}{3} \left(
		I[n,3,3,2] 
		\times \Bigl( 
		I[n-n_{3,3,2}, 1,1,1] (N_{3,3,2} N_{1,1,1} -1) \right. \\
		&& +I[n-n_{3,3,2}, 1,2,1] (N_{3,3,2} N_{1,2,1} -1) 
		+I[n-n_{3,3,2}, 1,3,1] (N_{3,3,2} N_{1,3,1} -1) 
		\Bigr) \\
	&& + I[n,3,4,2] 
		\times \Bigl( 
		I[n-n_{3,4,2}, 1,1,1] (N_{3,4,2} N_{1,1,1} -1) \\
		&& +I[n-n_{3,4,2}, 1,2,1] (N_{3,4,2} N_{1,2,1} -1) 
		+I[n-n_{3,4,2}, 1,3,1] (N_{3,4,2} N_{1,3,1} -1) 
		\Bigr) \\
	&& + I[n,3,5,2] 
		\times \Bigl( 
		I[n-n_{3,5,2}, 1,1,1] (N_{3,5,2} N_{1,1,1} -1) \\
		&& \left. +I[n-n_{3,5,2}, 1,2,1] (N_{3,5,2} N_{1,2,1} -1) 
		+I[n-n_{3,5,2}, 1,3,1] (N_{3,5,2} N_{1,3,1} -1) 
		\Bigr) 
	\right) \\
	&=& 16 \binom{n}{4} + 2960 \binom{n}{5} + 20960 \binom{n}{6} + 45360 \binom{n}{7} +33600 \binom{n}{8}.
\end{eqnarray*}
The number of moves of degree three for the partition $(2,2)$ of four is 
\begin{eqnarray*}
	&& \frac{1}{2!} \binom{4}{2} I[n,2,3,2] 
		\times \Bigl( 
		I[n-n_{2,3,2}, 2,2,1] (N_{2,3,2} N_{2,2,1} -1) \\
		&& + I[n-n_{2,3,2}, 2,3,1] (N_{2,3,2} N_{2,3,1} -1) 
		+ I[n-n_{2,3,2}, 2,3,2] (N_{2,3,2} N_{2,3,2} -1) \\
		&& + I[n-n_{2,3,2}, 2,4,1] (N_{2,3,2} N_{2,4,1} -1) 
		+ I[n-n_{2,3,2}, 2,5,1] (N_{2,3,2} N_{2,5,1} -1) 
		\Bigr) \\
	&=& 120 \binom{n}{5} + 3780 \binom{n}{6} + 12600 \binom{n}{7} + 10080 \binom{n}{8}.
\end{eqnarray*}
The number of moves of degree three for the partition $(2,1,1)$ of four is 
\begin{eqnarray*}
	&& \frac{1}{2!} \frac{4!}{2!1!1!} I[n,2,3,2] \\
	&&\times \Bigl( 
		I[n-n_{2,3,2}, 1,1,1] \bigl(
			I[n-n_{2,3,2}-n_{1,1,1}, 1,1,1] (N_{2,3,2} N_{1,1,1}N_{1,1,1} -1) \\
			&&+I[n-n_{2,3,2}-n_{1,1,1}, 1,2,1] (N_{2,3,2} N_{1,1,1}N_{1,2,1} -1) \\
			&&+I[n-n_{2,3,2}-n_{1,1,1}, 1,3,1] (N_{2,3,2} N_{1,1,1}N_{1,3,1} -1) 
			\bigr) \\
		&&+I[n-n_{2,3,2}, 1,2,1] \bigl(
			I[n-n_{2,3,2}-n_{1,2,1}, 1,1,1] (N_{2,3,2} N_{1,2,1}N_{1,1,1} -1) \\
			&&+I[n-n_{2,3,2}-n_{1,2,1}, 1,2,1] (N_{2,3,2} N_{1,2,1}N_{1,2,1} -1) \\
			&&+I[n-n_{2,3,2}-n_{1,2,1}, 1,3,1] (N_{2,3,2} N_{1,2,1}N_{1,3,1} -1) 
			\bigr) \\
		&&+I[n-n_{2,3,2}, 1,3,1] \bigl(
			I[n-n_{2,3,2}-n_{1,3,1}, 1,1,1] (N_{2,3,2} N_{1,3,1}N_{1,1,1} -1) \\
			&&+I[n-n_{2,3,2}-n_{1,3,1}, 1,2,1] (N_{2,3,2} N_{1,3,1}N_{1,2,1} -1) \\
			&&+I[n-n_{2,3,2}-n_{1,3,1}, 1,3,1] (N_{2,3,2} N_{1,3,1}N_{1,3,1} -1) 
			\bigr)
		\Bigr) \\
		&=& 120 \binom{n}{5} + 1440 \binom{n}{6} +6720 \binom{n}{7} + 13440 \binom{n}{8} + 10080 \binom{n}{9}.
\end{eqnarray*}
Then the number of moves of degree three for $r=4$ is
\begin{eqnarray*}
	160 \binom{n}{4} + 28040 \binom{n}{5} + 86660 \binom{n}{6} +102480 \binom{n}{7} + 57120 \binom{n}{8} + 10080 \binom{n}{9}.
\end{eqnarray*}

By the similar calculation we obtain the following polynomial which represents the number of moves of degree two for the case $r=5$:
\begin{eqnarray*}
	&& 28840 \binom{n}{5} + 6883200 \binom{n}{6} + 36009400 \binom{n}{7} + 83316800 \binom{n}{8} + 107898000 \binom{n}{9} \\
	&& + 76104000 \binom{n}{10} + 27720000 \binom{n}{11} + 3696000 \binom{n}{12}.
\end{eqnarray*}

The number of moves of degree three in minimal Markov bases are summarized as \tabref{table:number_of_three}.
The authors confirmed that the numbers above the horizontal lines in \tabref{table:number_of_three} coincides with the numbers obtained by the software 4ti2\cite{4ti2}.
\begin{table}[t]
\centering
\renewcommand{\tabcolsep}{5pt}
{\footnotesize
\caption{Number of moves of degree three.}
\label{table:number_of_three}
\begin{tabular}{|r||r|r|r|r|}
\hline
\backslashbox{$n$}{$r$} &2&3&4&5\\\hline\hline
1&0&0&0&0\\
2&0&0&0&0\\
3&1&1&0&0\\
4&4&160&160&0\\
5&10&1000&28840&28840\\\cline{4-5}
6&20&3680&257300&7056240\\
7&35&10325&1303540&84797440\\
8&56&24416&4884880&565736640\\\cline{3-3}
9&84&51240&15046080&2735910240\\
10&120&98400&40267080&10678207680\\\hline
\end{tabular}
}
\end{table}

\section{Discussion}
\label{sec:discussion}

In this section we discuss some topics related to our main result.

\subsection{Extension of fibers by allowing one negative element}
\label{subsec:negative}
As discussed after the proof of Lemma  \ref{lem:improper transposition lemma},~
we have shown the following result by our proof of ~Theorem \ref{thm:main} (cf.\ \eqref{eq:all move}).

\begin{proposition}
\label{prop:negative cell}
Let $P$ and $P^{\prime}$ be any two proper datasets with the same sufficient statistic. 
If we allow improper datasets as the intermediate states, $P$ and $P^{\prime}$ are connected by swap operations among two votes, whose single operation is used at each step.
\end{proposition}

Since an improper dataset has one $-1$, Proposition \ref{prop:negative cell} seems to suggest that every fiber ${\cal F}_\Bt$ for the configuration $A_{n,r}$ becomes connected by degree two moves if we extend ${\cal F}_\Bt$ by allowing one negative element $x(\sigma)=-1$ in $\Bx$ which satisfies $\Bt= A \Bx$.  
However, this is incorrect. 
In fact, allowing $-1$ in dataset and allowing $-1$ in ${\cal F}_\Bt$ are two different things.

For example, consider the case of $n=3$ and $r=2$ with candidates labeled as $a,b,c$. 
It is easy to see that $\dim \ker A_{3,2}=1$ and $I_{A_{3,2}}$ is a principal ideal generated by a single binomial $p(ab)p(bc)p(ca)-p(ac)(cb)p(ba)$.
Hence there is no degree two move in $\kerz A_{3,2}$.  Yet, we can connect two datasets
\[
P=
\left[
\begin{tabular}{cc}
$a$&$b$\\
$b$&$c$\\
$c$&$a$
\end{tabular}
\right]
, \qquad
P^{\prime}=
\left[\begin{tabular}{cc}
$a$&$c$\\
$c$&$b$\\
$b$&$a$
\end{tabular}
\right]
\]
by applying $\{2,3\}:a\swapj{1} b $, \  $\{1,2\}: b \swapj{2} c$ and 
$\{2,3\}: a\swapj{1} c$  in this order:
\[
\left[
\begin{tabular}{cc}
$a$&$b$\\
$b$&$c$\\
$c$&$a$
\end{tabular} 
\right]
\rightarrow
\left[
\begin{tabular}{cc}
$a$&$b$\\
$a$&$c$\\
$b+c-a$&$a$
\end{tabular} 
\right]
\rightarrow
\left[\begin{tabular}{cc}
$a$&$c$\\
$a$&$b$\\
$b+c-a$&$a$
\end{tabular}
\right]
\rightarrow
\left[\begin{tabular}{cc}
$a$&$c$\\
$c$&$b$\\
$b$&$a$
\end{tabular}
\right].
\]
Note that the middle two datasets can be interpreted either as
\[
\text{adding } (ab),(ac),(bc),(ca) \text{ and subtracting } (ac)
\]
or as
\[
\text{adding } (ab),(ac),(ba),(cb) \text{ and subtracting } (ab).
\]
However the middle two datasets do not correspond to an element of a fiber for $A_{3,2}$.

\subsection{Normality}
\label{subsec:normality}
It is natural to ask if the semigroup generated by $A_{n,r}$ is normal.
Consider the set $Q$ of $n\times r$ real matrices $X=\{x_{ij}\}$ satisfying
\[
0\le x_{ij} \le 1, \ \forall i,j,  \qquad \sum_{j=1}^r x_{ij}\le 1, \forall i, \qquad
\sum_{i=1}^n x_{ij} =1, \forall j.
\]
The set $Q$ is the Birkhoff polytope in ${\mathbb R}^{n\times r}$, which is a special case of transportation polytopes\cite{haase-paffenholz-2009}.
By \cite{brualdi-dahl} the set of vertices of $Q$ is exactly the same as the set of columns
of $A_{n,r}$.  
Then by the results of \cite{ohsugi-hibi-2001} and \cite{sullivant-2006} the semigroup generated by $A_{n,r}$ is normal.
Lemma \ref{lem:collision lemma} is a consequence of this normality, because by the normality 
there exists two valid votes with the same sufficient statistic as the $i$-th and the $i'$-th votes in Lemma \ref{lem:collision lemma}.
These two proper votes can be obtained from the two votes of a swap operation
in Lemma \ref{lem:collision lemma}.
However, the normality is not useful in proving Lemmas \ref{lem:extended-1} and \ref{lem:extended-2}.

\subsection{Generation of moves for running a Markov chain}
\label{subsec:MCMC}
Based on Theorem \ref{thm:main} we can run a Markov chain for general $r$ as follows.
We randomly generate two or three proper votes of $r$ candidates out of $n$ candidates.
Once these votes are obtained, we randomly perform permutations of candidates in the same position.
We do this for each position.  If no collision occurs, then we have two or three proper votes.
If the obtained set of votes is different from the initial set, then the difference is a move.  
In this way, we obtain a random move of degree two or three and then run a Markov chain over a given fiber.

\section*{Acknowledgment}
This work is partially supported by Grant-in-Aid for JSPS Fellows (No. 12J07561) from Japan Society for the Promotion of Science (JSPS).
The authors thank to the anonymous referees for the valuable comments.

\bibliographystyle{abbrv}
\bibliography{birkhoff}

\end{document}